\pdfoutput=1
\documentclass[reqno]{amsart}
\usepackage{enumerate}
\usepackage{tabto}
\usepackage[mathscr]{euscript}
\usepackage{xcolor}
\usepackage{layout}
\usepackage{fancyhdr}
\usepackage{array}
\usepackage{amsfonts}
\usepackage{amsmath}
\usepackage{amssymb}
\usepackage{mathtools}
\usepackage{graphicx}
\usepackage{bm}
\usepackage{enumitem}
\usepackage{caption}
\usepackage{color}
\usepackage{csquotes}
\usepackage{bookmark}
\usepackage{float}
\usepackage{multirow}
\usepackage[square,numbers,sort&compress]{natbib}
\usepackage{hyperref}
\hypersetup{colorlinks=true,linkcolor=blue,citecolor=red}
\allowdisplaybreaks

\def\Xint#1{\mathchoice
{\XXint\displaystyle\textstyle{#1}}%
{\XXint\textstyle\scriptstyle{#1}}%
{\XXint\scriptstyle\scriptscriptstyle{#1}}%
{\XXint\scriptscriptstyle\scriptscriptstyle{#1}}%
\!\int}
\def\XXint#1#2#3{{\setbox0=\hbox{$#1{#2#3}{\int}$ }
\vcenter{\hbox{$#2#3$ }}\kern-.6\wd0}}

\def\dashint{\Xint-}

\newtheorem{theorem}{Theorem}[section]
\newtheorem{lemma}[theorem]{Lemma}

\newtheorem{remark}[theorem]{Remark}
\theoremstyle{definition}
\newtheorem{definition}[theorem]{Definition}

\numberwithin{equation}{section}

\newcommand{ \mr }{ \mathbb{R} }

\newcommand{\iints}[1]{{\int\hspace{-0.28cm}\int_{#1}}}
\newcommand{\iintss}{{\int\hspace{-0.28cm}\int}}
\newcommand{ \miints }{{\iintss\hspace{-0.56cm} -\hspace{-0.15cm}-}}
\newcommand{\miint}[1]{{\miints_{\hspace{-0.13cm}#1}}}

\begin{document}
\title[The H\"older interpolative gap bound]{The H\"older interpolative gap bound for degenerate parabolic double phase problems}

\author{Jehan Oh}\address{Department of Mathematics, Kyungpook National University, Daegu, 41566, Republic of Korea} \email{jehan.oh@knu.ac.kr}

\date{\today.}

\begin{abstract}
We study weak solutions to degenerate parabolic double phase equations with growth exponents $2\leq p<q$ and a modulating coefficient that is H\"older continuous with exponent $\alpha$. If the solution itself is H\"older continuous with exponent $\gamma$, we prove higher integrability of its gradient under the gap condition $q\leq p+\alpha/(1-\gamma)$ together with $q<p+1$. This is the first gap bound of H\"older interpolative type for parabolic double phase problems, and it is the parabolic counterpart of the corresponding bound for elliptic problems. For $\alpha<1$ it allows exponents beyond the range $q\leq p+\alpha$ known for bounded solutions.
\end{abstract}
\maketitle

\section{\bf Introduction}\label{section 1}
We study parabolic double phase equations of the form
\begin{equation}    \label{eq : main equation}
    u_t - \operatorname{div} \mathcal{A}(z,Du)=0 \qquad \text{in} \ \, \Omega_T\coloneq \Omega\times (0,T),
\end{equation}
where $\Omega \subset \mr^n$ ($n\geq 2$) is a bounded open set and $\mathcal{A}:\Omega_T \times \mr^{n}\rightarrow \mr^{n}$ is a Carath\'{e}odory vector field satisfying the following coercivity and growth bounds: there exist constants $0<\nu\leq L <\infty$ such that
\begin{equation}    \label{cond : double phase bounded condition of integrand}
    \mathcal{A}(z,\xi)\cdot \xi\geq \nu (|\xi|^p+a(z)|\xi|^q)\quad \text{and}\quad |\mathcal{A}(z,\xi)|\leq L(|\xi|^{p-1}+a(z)|\xi|^{q-1})
\end{equation}
for all $z \in \Omega_T$ and $\xi \in \mr^n$, where $2\leq p <q <\infty$.
We note that the corresponding model equation is
\begin{equation}    \label{eq : model equation}
u_t-\operatorname{div} \left(|Du|^{p-2}Du + a(z)|Du|^{q-2}Du\right)=0 \qquad \text{in} \ \, \Omega_T.
\end{equation}

The parabolic equation \eqref{eq : model equation} is the evolutionary counterpart of the elliptic double phase functional
$$
W^{1,1}(\Omega) \ni w \mapsto \mathcal{P}(w,\Omega)\coloneq \int_\Omega \left[\frac{1}{p}|Dw|^p+\frac{1}{q}a(x)|Dw|^q\right] dx,
$$
where $1<p\leq q$ and $0\leq a(\cdot)\in C^{\alpha}(\Omega)$ for some $\alpha\in(0,1]$. The elliptic double phase model was originally introduced in \cite{Zhikov1986,Zhikov1993,Zhikov1995,Zhikov1997} as an example exhibiting the Lavrentiev phenomenon and illustrating homogenization in strongly anisotropic materials, and variants of the double phase problem arise naturally in applied sciences, including transonic flows \cite{Bahrouni2019}, quantum physics \cite{Benci2000}, stationary reaction-diffusion systems \cite{cherfils2005stationary}, image denoising and processing \cite{Kbiri2014,Charkaoui2024,Chen2006,Harjulehto2013,Harjulehto2021,Fang2010}, and heat diffusion in materials with heterogeneous thermal properties \cite{Arora2023}.

For the elliptic problem, which falls into the class of problems with $(p,q)$-growth initiated in the seminal papers of Marcellini \cite{Marcellini1989,Marcellini1991}, regularity of weak solutions is governed by conditions linking the closeness of $p$ and $q$ with the H\"{o}lder exponent $\alpha$ of the modulating coefficient $a(\cdot)$, see \cite{Colombo2015a,Mingione2021}. Under the gap bound condition
\begin{equation}\label{cond : dimension gap bound condition in elliptic double phase problem}
\frac{q}{p}\leq 1+\frac{\alpha}{n},
\end{equation}
it has been shown in \cite{Baroni2018,Colombo2015,Esposito2004} that a weak solution $u$ and its gradient $Du$ are H\"{o}lder continuous. When $u$ satisfies
\begin{equation}\label{cond : bounded gap bound condition in elliptic double phase problem}
u\in L^\infty(\Omega)\quad \text{and}\quad q\leq p +\alpha,
\end{equation}
the same conclusion has been established in \cite{Baroni2018,Colombo2015a}. Moreover, Baroni-Colombo-Mingione \cite{Baroni2018} have proved that if
\begin{equation}\label{cond : Holder gap bound condition in elliptic double phase problem}
u\in C^{0,\gamma}(\Omega) \quad \text{and}\quad q<p +\frac{\alpha}{1-\gamma}\ \, \text{with}\ \, \gamma\in(0,1),
\end{equation}
then the gradient of $u$ is H\"{o}lder continuous. The inequality in this gradient regularity result is strict, while the non-strict bound appears in the result on the absence of the Lavrentiev phenomenon in \cite[Theorem 4]{Baroni2018}. This shows that by imposing stronger regularity assumptions on $u$, one may relax the gap bound condition while still ensuring regularity results for $u$. In \cite{Baroni2018} this interpolative type phenomenon was in fact conjectured to be a general principle for problems with $(p,q)$-growth. In the same spirit, an interpolation of the gap bounds with respect to integrability assumptions $u\in L^\gamma_{\operatorname{loc}}(\Omega)$ was obtained in \cite{Ok2020}. Furthermore, under either \eqref{cond : dimension gap bound condition in elliptic double phase problem} or \eqref{cond : bounded gap bound condition in elliptic double phase problem}, various regularity results have been established, including Harnack inequalities and H\"older continuity \cite{Baroni2015,Ok2017,Ok2020} as well as Calder\'{o}n-Zygmund type estimates \cite{Baasandorj2020,DeFilippis2019,Colombo2016}; see also \cite{Byun2021,Byun2017,Byun2020,Byun2021a,Haestoe2022,Haestoe2022a}.

In the parabolic setting, Kim-Kinnunen-Moring \cite{2023_Gradient_Higher_Integrability_for_Degenerate_Parabolic_Double-Phase_Systems} and Kim-Kinnunen-S\"{a}rki\"{o} \cite{Wontae2023a} established gradient higher integrability and energy estimates for weak solutions under the conditions
\begin{equation}\label{cond : s=2 gap bound condition in parabolic double phase problem}
    u\in C(0,T;L^2(\Omega)),\quad a(\cdot) \in C^{\alpha,\frac{\alpha}{2}}(\Omega_T) \quad \text{and}\quad q\leq p + \frac{2\alpha}{n+2}.
\end{equation}
Here, $a(\cdot) \in C^{\alpha,\frac{\alpha}{2}}(\Omega_T)$ means that $a(\cdot)\in L^\infty (\Omega_T)$ and there exists a H\"{o}lder constant $[a]_\alpha\coloneq [a]_{\alpha,\frac{\alpha}{2};\Omega_T}>0$ such that
$$
|a(x_1,t_1)-a(x_2,t_2)|\leq [a]_\alpha \max\left\{|x_1-x_2|^\alpha,|t_1-t_2|^\frac{\alpha}{2}\right\}
$$
for all $x_1,\,x_2\in\Omega$ and $t_1,\,t_2\in(0,T)$. The existence theory of weak solutions to the above problem is addressed in \cite{Chlebicks2019}, with further discussions in \cite{Wontae2023a,Singer2016}. Kim \cite{Wontae2023b} and Kim-S\"{a}rki\"{o} \cite{Wontae2024} have investigated higher integrability results and Calder\'{o}n-Zygmund type estimates for singular parabolic double phase systems, Buryachenko-Skrypnik \cite{Buryachenko2022} have established local continuity and Harnack's inequality for the parabolic double phase equations, and regularity results for parabolic multi-phase problems have been studied in \cite{Kim2024,Kim2025,Sen2025}.

Recently, in the companion work of Kim and the author \cite{KimOh2026}, the gradient higher integrability result was proved for bounded weak solutions under the dimensionless gap condition
\begin{equation}\label{cond : bounded gap bound condition in parabolic double phase problem}
u\in L^\infty(\Omega_T)\quad \text{and}\quad q\leq p+\alpha,
\end{equation}
together with the interpolative family of conditions $u\in C(0,T;L^s(\Omega))$ and $q\leq p+\frac{s\alpha}{n+s}$ for $s\in[2,\infty)$, which connects \eqref{cond : s=2 gap bound condition in parabolic double phase problem} and \eqref{cond : bounded gap bound condition in parabolic double phase problem}. Chlebicka-Garain-Kim \cite{Chlebicka2025} obtained the same range for bounded solutions to parabolic double phase systems. We also mention that Kim-Moring-S\"{a}rki\"{o} \cite{Wontae2025} have established that, under \eqref{cond : bounded gap bound condition in parabolic double phase problem}, bounded weak solutions are locally H\"{o}lder continuous.

In contrast with the elliptic theory, all gap bounds available in the parabolic setting so far are of integrability type: the conditions of \cite{2023_Gradient_Higher_Integrability_for_Degenerate_Parabolic_Double-Phase_Systems,Wontae2023a,KimOh2026,Chlebicka2025} interpolate between $q\leq p+\frac{2\alpha}{n+2}$ and the dimensionless bound $q\leq p+\alpha$, and no gap condition in which the admissible range of $q$ grows with the regularity of the solution beyond boundedness has been known. In particular, the H\"older interpolative bound \eqref{cond : Holder gap bound condition in elliptic double phase problem} had no parabolic counterpart. This is more than a gap in the literature: already in \cite[Remark 2]{Baroni2018} the parabolic equation \eqref{eq : model equation} was singled out as posing non-trivial additional difficulties, as it generates new double phase intrinsic geometries for which new methods must be developed.

The purpose of this paper is to establish the first gap bound of H\"older interpolative type in the parabolic setting. We prove that if
$$
u\in C^{\gamma,\frac{\gamma}{2}}(\Omega_T) \quad \text{and}\quad q\leq p+\frac{\alpha}{1-\gamma}\ \, \text{with}\ \, \gamma\in(0,1),
$$
together with the restriction $q<p+1$ discussed in Remark \ref{rem : role of q<p+1}, then the gradient of a weak solution to \eqref{eq : main equation} enjoys higher integrability. To the best of our knowledge, this is the first gap bound condition for the parabolic double phase problem that exploits H\"older continuity of the solution. We therefore regard the identification of the H\"older interpolative gap bound under the parabolic intrinsic geometries, rather than the higher integrability estimate itself, as the main novelty of this paper: it confirms that the interpolation principle conjectured in \cite{Baroni2018} persists in the parabolic setting, and it extends the parabolic family of gap bounds initiated in \cite{2023_Gradient_Higher_Integrability_for_Degenerate_Parabolic_Double-Phase_Systems,KimOh2026}, see Remark \ref{rem : ladder}. Higher integrability is here the natural first regularity result to establish under the new condition, being the pivotal starting point of the double phase theory both in the elliptic case \cite{Esposito2004} and in the parabolic case \cite{2023_Gradient_Higher_Integrability_for_Degenerate_Parabolic_Double-Phase_Systems,KimOh2026}.

We say that $u\in C^{\gamma,\frac{\gamma}{2}}(\Omega_T)$ for $\gamma\in(0,1)$ if $u\in L^\infty(\Omega_T)$ and there exists a H\"older constant $[u]_\gamma\coloneq [u]_{\gamma,\frac{\gamma}{2};\Omega_T}>0$ such that
$$
|u(x_1,t_1)-u(x_2,t_2)|\leq [u]_\gamma \max\left\{|x_1-x_2|^\gamma,|t_1-t_2|^\frac{\gamma}{2}\right\}
$$
for all $x_1,\,x_2\in\Omega$ and $t_1,\,t_2\in(0,T)$, and we write $\|u\|_{C^{\gamma,\frac{\gamma}{2}}(\Omega_T)}\coloneq \|u\|_{L^\infty(\Omega_T)}+[u]_\gamma$.

We now introduce the definition of a weak solution for the parabolic double phase problem.
\begin{definition}
    A function $u:\Omega_T\rightarrow \mr$ with
    $$
    u\in C(0,T;L^2(\Omega))\cap L^q(0,T;W^{1,q}(\Omega))
    $$
    is a weak solution to \eqref{eq : main equation} if
    $$
    \iints{\Omega_T} -u\cdot \varphi_t +\mathcal{A}(z,Du)\cdot D\varphi\, dz=0
    $$
    for every $\varphi\in C_0^\infty(\Omega_T)$.
\end{definition}

The theorem below uses the weak-solution class in the preceding definition, including the assumption $u\in L^q(0,T;W^{1,q}(\Omega))$. The natural energy class would instead require
$$
u\in C(0,T;L^2(\Omega))\cap L^1(0,T;W^{1,1}(\Omega)),\qquad
\iints{\Omega_T} (|Du|^p+a(z)|Du|^q)\,dz<\infty.
$$
Parabolic Lipschitz truncation provides energy estimates in this class under the gap assumptions of \cite{Wontae2023a}. Those assumptions do not cover the entire H\"older interpolative range considered here. Extending the present theorem to all solutions in the natural energy class requires a separate justification of the energy estimates in that wider range and is not asserted in this paper.

Throughout this paper we assume that the non-negative coefficient function $a:\Omega_T\rightarrow \mr^+$ and a weak solution $u$ satisfy
\begin{equation}    \label{cond : main assumption with gamma}
    u\in C^{\gamma,\frac{\gamma}{2}}(\Omega_T),\quad a(\cdot) \in C^{\alpha,\frac{\alpha}{2}}(\Omega_T),\quad q\leq p+\frac{\alpha}{1-\gamma}\quad \text{and}\quad q<p+1
\end{equation}
for some $\gamma\in(0,1)$ and $\alpha\in(0,1]$. In what follows, $Q_r(z_0)\coloneq B_r(x_0)\times (t_0-r^2,t_0+r^2)$ denotes a parabolic cylinder, where $z_0=(x_0,t_0)$. For simplicity, we write the collection of parameters as
$$
\operatorname{data}_\gamma \coloneq (n,p,q,\alpha,\gamma,\nu,L,\operatorname{diam}(\Omega),[a]_\alpha,\|u\|_{C^{\gamma,\frac{\gamma}{2}}(\Omega_T)}),
$$
and we denote $H(z,\varkappa) \coloneq \varkappa^p +a(z)\varkappa^q$ for $\varkappa\geq 0$ and $z\in\Omega_T$.

Our main result is the following gradient higher integrability theorem.
\begin{theorem}\label{thm : main theorem with gamma}
    Assume that \eqref{cond : main assumption with gamma} is satisfied and let $u$ be a weak solution to \eqref{eq : main equation} with \eqref{cond : double phase bounded condition of integrand}. Then there exist constants $\varepsilon_0=\varepsilon_0(\operatorname{data}_\gamma)>0$ and $c=c(\operatorname{data}_\gamma,$ $\|a\|_{L^\infty(\Omega_T)})>1$ such that
    $$
    \begin{aligned}
        &\miint{Q_r(z_0)} H(z,|Du|)^{1+\varepsilon}\, dz\leq c\left(\miint{Q_{2r}(z_0)} H(z,|Du|)\,dz\right)^{1+\frac{\varepsilon q}{2}}+c
    \end{aligned}
    $$
    for every $Q_{2r}(z_0)\subset \Omega_T$ and $\varepsilon\in (0,\varepsilon_0)$.
\end{theorem}

\begin{remark}\label{rem : ladder}
    The admissible range in Theorem \ref{thm : main theorem with gamma} is
    $$
    0<q-p\leq\frac{\alpha}{1-\gamma},\qquad q-p<1.
    $$
    For $0<\alpha<1$, this contains the bounded-solution range $q\leq p+\alpha$ and permits additional exponents $q>p+\alpha$. For $\alpha=1$, however, the restriction $q<p+1$ prevents any enlargement and excludes the endpoint $q=p+1$ allowed by the bounded-solution result. As $\gamma\searrow0$, the numerical bound tends to $p+\alpha$. This is only a formal comparison, since the theorem still assumes H\"older continuity of $u$.

    Thus the family of parabolic bounds progresses from $q\leq p+\frac{2\alpha}{n+2}$ for $u\in C(0,T;L^2)$, through $q\leq p+\frac{s\alpha}{n+s}$ for $u\in C(0,T;L^s)$ and $q\leq p+\alpha$ for bounded solutions, to the above H\"older interpolative range, with its additional cap $q<p+1$. The quantity $p+\alpha/(1-\gamma)$ agrees with the elliptic interpolative threshold in \eqref{cond : Holder gap bound condition in elliptic double phase problem}. The cited elliptic gradient H\"older theorem requires a strict inequality, whereas our higher integrability estimate also allows equality at this threshold when $\alpha/(1-\gamma)<1$.
\end{remark}

\begin{remark}\label{rem : role of q<p+1}
    Let us comment on the last condition $q<p+1$ in \eqref{cond : main assumption with gamma}. It is needed at exactly one point of the proof: in the $(p,q)$-intrinsic parabolic Poincar\'{e} inequalities of Section \ref{section 4}, the mean drift coming from the gluing lemma is inserted into the structure function $H_{z_0}$, and the unweighted $p$-component of the resulting expression carries the integral average of $|Du|^{q-1}$, which can be absorbed into the $\theta$-reduced energy only through H\"{o}lder's inequality with exponent $\frac{\theta p}{q-1}\geq 1$. This forces $\theta\geq\frac{q-1}{p}$, while the reverse H\"{o}lder inequalities of Section \ref{section 5} require $\theta<1$, whence $q<p+1$. On the other hand, the $p$-phase estimates of Section \ref{section 4} hold for every $\theta\in(\frac{q-1}{q},1]$ and do not require this condition. We emphasize that the $(p,q)$-phase estimates do not involve the assumption \eqref{cond : main assumption with gamma}$_1$ on $u$, so that the condition $q<p+1$ is inherited from the underlying machinery of \cite{2023_Gradient_Higher_Integrability_for_Degenerate_Parabolic_Double-Phase_Systems,KimOh2026} rather than caused by the H\"{o}lder continuity assumption. In \cite{2023_Gradient_Higher_Integrability_for_Degenerate_Parabolic_Double-Phase_Systems} it holds automatically, since $q\leq p+\frac{2\alpha}{n+2}<p+1$. In the elliptic case \cite{Baroni2018} no such restriction appears. We do not know whether it can be removed in the parabolic setting.
\end{remark}

Let us briefly describe the strategy of the proof and the main new ingredients. In \cite{Baroni2018} the elliptic result under \eqref{cond : Holder gap bound condition in elliptic double phase problem} is obtained through a blow-up argument relying on a nonlinear harmonic type approximation lemma, and no parabolic version of this machinery adapted to the double phase intrinsic geometries is available at present. We follow instead the intrinsic scaling approach of \cite{2023_Gradient_Higher_Integrability_for_Degenerate_Parabolic_Double-Phase_Systems} in the version developed in \cite{KimOh2026}, which builds on the intrinsic geometry of DiBenedetto \cite{1993_Degenerate_parabolic_equations_DiBenedetto} and its use in the gradient higher integrability results of \cite{Acerbi2004,Kinnunen2000}: in the stopping time argument of Section \ref{section 3}, we distinguish the $p$-intrinsic and $(p,q)$-intrinsic cases by imposing
$$
K\lambda^p\geq \sup a(\cdot) \lambda^q \quad \text{and}\quad K\lambda^p \leq \sup a(\cdot) \lambda^q,
$$
respectively, for some $K>1$, and we must exclude the borderline third case in which $K\lambda^p\leq \sup a(\cdot)\lambda^q$ while $\sup a(\cdot)\lesssim \rho^\alpha$. The key observation of the present paper is that on every intrinsic cylinder appearing in this scheme the time scale is at most $\rho^2$, since $\lambda\geq 1$ and $p\geq 2$, so that the H\"older continuity of $u$ yields the oscillation estimate
$$
\sup_{z\in Q^\lambda_\rho(z_0)}\Big|u(z)-u_{Q^\lambda_\rho(z_0)}\Big|\leq 2[u]_\gamma \rho^\gamma,
$$
see Lemma \ref{lem : oscillation on intrinsic cylinders}. Wherever the argument for bounded solutions in \cite{KimOh2026} invokes $\|u\|_{L^\infty(\Omega_T)}$, we invoke this oscillation bound instead and gain a factor $\rho^\gamma$ for each factor of $u$. In the exclusion of the third case this leads to the estimate
$$
\varrho_w\leq c\,\lambda_w^{-\frac{1}{1-\gamma}},
$$
which improves the decay $\varrho_w\lesssim \lambda_w^{-1}$ of the bounded case and is compatible with the gap condition $q\leq p+\frac{\alpha}{1-\gamma}$, see Lemma \ref{lem : no occurence with gamma}. In the reverse H\"{o}lder inequalities of Section \ref{section 5}, the same oscillation bound handles the term carrying the factor $\rho^\alpha$ directly, so that, in contrast to the interpolative $L^s$-track of \cite{KimOh2026}, no Gagliardo-Nirenberg argument with case distinctions in $s$ is needed for this term. The remaining machinery, that is, the parabolic Sobolev-Poincar\'{e} inequalities of Section \ref{section 4}, the $(p,q)$-phase estimates and the Vitali covering argument of Section \ref{section 6}, does not involve the assumption on $u$ and is taken from \cite{2023_Gradient_Higher_Integrability_for_Degenerate_Parabolic_Double-Phase_Systems,KimOh2026}.

The paper is organized as follows. In Section \ref{section 2}, we fix notation, prove the oscillation estimate on intrinsic cylinders and recall three energy lemmas from \cite{2023_Gradient_Higher_Integrability_for_Degenerate_Parabolic_Double-Phase_Systems}. Section \ref{section 3} contains the stopping time argument and shows that the third case never occurs. In Section \ref{section 4}, we collect parabolic Sobolev-Poincar\'{e} type inequalities in the $p$-phase and the $(p,q)$-phase, and Section \ref{section 5} is devoted to reverse H\"{o}lder inequalities in both phases. Finally, Section \ref{section 6} contains a Vitali type covering argument and the proof of Theorem \ref{thm : main theorem with gamma}.

\section{\bf Preliminaries}\label{section 2}
We first set up notation.
For a fixed point $z_0 \in \Omega_T$, we denote
\begin{equation}\label{def : definition of H with a fixed center z_0}
    H_{z_0}(\varkappa)\coloneq \varkappa^p+a(z_0)\varkappa^q \qquad \text{for } \varkappa\geq 0.
\end{equation}
We set intrinsic cylinders
\begin{equation}\label{eq : definition of p-intrinsic cylinder}
Q_\rho^\lambda(z_0)\coloneq  B_\rho(x_0)\times(t_0-\lambda^{2-p}\rho^2,t_0+\lambda^{2-p}\rho^2)\eqcolon B_\rho(x_0)\times I_{\rho}^\lambda(t_0)
\end{equation}
and
\begin{equation}\label{eq : definition of p,q-intrinsic cylinder}
G_\rho^\lambda(z_0)\coloneq  B_\rho(x_0)\times\left(t_0-\frac{\lambda^2}{H_{z_0}(\lambda)}\rho^2,t_0+\frac{\lambda^2}{H_{z_0}(\lambda)}\rho^2\right)\eqcolon B_\rho(x_0)\times J_\rho^\lambda(t_0).
\end{equation}
Next, we denote the super-level set by
\begin{equation}\label{def : definition of Psi}
    \Psi(\Lambda)\coloneq \{z\in\Omega_T:H(z,|Du(z)|)>\Lambda\}.
\end{equation}

For any measurable set $E\subset \Omega_T$ with $0<|E|<\infty$ and integrable function $f\in L^1(\Omega_T)$, we denote the integral average of $f$ over $E$ by
$$
f_E=\frac{1}{|E|}\iints{E} f\, dz=\miint{E} f \, dz,
$$
where $|E|$ means $(n+1)$-dimensional Lebesgue measure of the set $E$.

The following elementary oscillation estimate on intrinsic cylinders is the main new ingredient of this paper. Note that the time length of $Q^\lambda_\rho(z_0)$ is $2\lambda^{2-p}\rho^2$, which does not exceed $2\rho^2$ whenever $\lambda\geq 1$, since $p\geq 2$. The same is true for $G^\lambda_\rho(z_0)$, since $H_{z_0}(\lambda)\geq \lambda^p$ implies $J^\lambda_\rho(t_0)\subset I^\lambda_\rho(t_0)$.
\begin{lemma}\label{lem : oscillation on intrinsic cylinders}
    Let $u\in C^{\gamma,\frac{\gamma}{2}}(\Omega_T)$ for some $\gamma\in(0,1)$ and let $\lambda\geq 1$. Then, for every intrinsic cylinder $Q^\lambda_\rho(z_0)\subset \Omega_T$ as in \eqref{eq : definition of p-intrinsic cylinder}, we have
    $$
    \sup_{z\in Q^\lambda_\rho(z_0)}\Big|u(z)-u_{Q^\lambda_\rho(z_0)}\Big|\leq 2[u]_\gamma\, \rho^\gamma.
    $$
    The same estimate holds with $G^\lambda_\rho(z_0)$ in place of $Q^\lambda_\rho(z_0)$.
\end{lemma}
\begin{proof}
    Let $z_1=(x_1,t_1),\,z_2=(x_2,t_2)\in Q_\rho^\lambda(z_0)$. Then $|x_1-x_2|<2\rho$ and, since $\lambda\geq 1$ and $p\geq 2$,
    $$
    |t_1-t_2|<2\lambda^{2-p}\rho^2\leq 2\rho^2.
    $$
    Hence
    $$
    |u(z_1)-u(z_2)|\leq [u]_\gamma \max\left\{(2\rho)^\gamma, (2\rho^2)^{\frac{\gamma}{2}}\right\}=[u]_\gamma \max\left\{2^\gamma,2^{\frac{\gamma}{2}}\right\}\rho^\gamma\leq 2[u]_\gamma \rho^\gamma.
    $$
    Averaging over $z_2\in Q^\lambda_\rho(z_0)$ yields the claim. Since $J^\lambda_\rho(t_0)\subset I^\lambda_\rho(t_0)$, the same argument applies to $G^\lambda_\rho(z_0)$.
\end{proof}

To prove the main theorem, we refer to three energy lemmas from \cite{2023_Gradient_Higher_Integrability_for_Degenerate_Parabolic_Double-Phase_Systems}. The following lemma provides a Caccioppoli type inequality, see \cite[Lemma 2.3]{2023_Gradient_Higher_Integrability_for_Degenerate_Parabolic_Double-Phase_Systems}.
\begin{lemma}\label{lem : Caccioppoli inequality}
    Let $u$ be a weak solution to \eqref{eq : main equation}. Then there exists a positive constant $c=c(n,p,q,\nu,L)$ such that
    $$
    \begin{aligned}
        &\sup_{t\in (t_0-\tau,t_0+\tau)}\dashint_{B_r(x_0)} \frac{|u-u_{Q_{r,\tau}(z_0)}|^2}{\tau}\, dx+\miint{Q_{r,\tau}(z_0)} H(z,|Du|)\, dz\\
        &\quad \leq c \miint{Q_{R, \ell}\left(z_0\right)}\left[\frac{\left|u-u_{Q_{R, \ell}\left(z_0\right)}\right|^p}{(R-r)^p}+a(z) \frac{\left|u-u_{Q_{R, \ell}\left(z_0\right)}\right|^q}{(R-r)^q}\right] dz \\
        &\qquad+c \miint{Q_{R, \ell}\left(z_0\right)} \frac{\left|u-u_{Q_{R, \ell}\left(z_0\right)}\right|^2}{\ell-\tau} \, dz
    \end{aligned}
    $$
    for every $Q_{R, \ell}\left(z_0\right)=B_R(x_0) \times(t_0-\ell, t_0+\ell) \subset \Omega_T$, with $R, \ell>0,\, r \in[R / 2, R)$ and $\tau \in [\ell / 2^2, \ell )$.
\end{lemma}
The next lemma is a gluing lemma, see \cite[Lemma 2.4]{2023_Gradient_Higher_Integrability_for_Degenerate_Parabolic_Double-Phase_Systems}. For this, we denote the spatial integral average of $u$ over $B_R(x_0)$ by
$$
(u)_{B_R(x_0)}(t)=\dashint_{B_R(x_0)} u(x,t)\, dx.
$$
\begin{lemma}\label{lem : gluing lemma}
    Let $u$ be a weak solution to \eqref{eq : main equation}, and let $\eta \in C_0^{\infty}(B_R(x_0))$ be a function such that
    $$
    \eta \geq 0, \quad \dashint_{B_R(x_0)} \eta \, dx=1 \quad \text { and } \quad \|\eta\|_{L^{\infty}}+R\|D \eta\|_{L^{\infty}} \leq c(n).
    $$
    Then there exists a positive constant $c=c(n, L)$ such that
\begin{multline*}
    \sup_{t_1, t_2 \in (t_0-\ell, t_0+\ell)} |(u \eta)_{B_R(x_0)}(t_2)-(u \eta)_{B_R(x_0)}(t_1)| \\
    \leq c \frac{\ell}{R} \miint{Q_{R, \ell}(z_0)}\left[|D u|^{p-1}+a(z)|D u|^{q-1}\right] dz
\end{multline*}
    for every $Q_{R, \ell}(z_0) \subset \Omega_T$ with $R, \ell>0$.
\end{lemma}
The above lemma yields a parabolic Poincar\'{e} inequality, see \cite[Lemma 2.5]{2023_Gradient_Higher_Integrability_for_Degenerate_Parabolic_Double-Phase_Systems}.
\begin{lemma}\label{lem : semi-Parabolic Poincare inequality}
    Let $u$ be a weak solution to \eqref{eq : main equation}. Then there exists a positive constant $c=c(n, m, L)$ such that
\begin{multline*}
        \miint{Q_{R, \ell}(z_0)} \frac{|u-u_{Q_{R, \ell}(z_0)}|^{\theta m}}{R^{\theta m}} \,dz \\
        \leq c \miint{Q_{R, \ell}(z_0)}|D u|^{\theta m} \,dz+c\left(\frac{\ell}{R^2} \miint{Q_{R, \ell}\left(z_0\right)} \left[|D u|^{p-1}+a(z)|D u|^{q-1}\right] d z\right)^{\theta m}
\end{multline*}
    for every $Q_{R, \ell}(z_0)=B_R(x_0) \times(t_0-\ell, t_0+\ell) \subset \Omega_T$ with $R, \ell>0, m \in(1, q]$ and $\theta \in(1 / m, 1]$.
\end{lemma}

\section{\bf Stopping time argument}\label{section 3}
We put
\begin{equation}\label{def : lambda_0 and Lambda_0}
\lambda_0^2\coloneq \miint{Q_{2r}(z_0)} \left[H(z,|Du|)+1\right] dz\quad \text{and}\quad \Lambda_0\coloneq \lambda_0^p+\sup_{Q_{2r}(z_0)} a(\cdot)\lambda_0^q,
\end{equation}
where $Q_{2r}(z_0)=B_{2r}(x_0)\times (t_0-(2r)^2,t_0+(2r)^2)$. Let
\begin{equation}\label{def : definition of K and kappa}
    K\coloneq 1+80c_\gamma[a]_\alpha \quad \text{and}\quad \kappa\coloneq 10K,
\end{equation}
where $c_\gamma$ will be defined in Lemma \ref{lem : no occurence with gamma}.
For $\Psi(\Lambda)$ as in \eqref{def : definition of Psi} and $\varrho\in [r,2r]$, we write
$$
\Psi(\Lambda,\varrho)\coloneq  \Psi(\Lambda)\cap Q_\varrho(z_0)=\{z\in Q_\varrho(z_0):H(z,|Du(z)|)>\Lambda\}.
$$

We now apply a stopping time argument. Let $r\leq r_1<r_2\leq 2r$ and
$$
\Lambda>\left(\frac{4\kappa r}{r_2-r_1}\right)^{\frac{q(n+2)}{2}}\Lambda_0,
$$
where $\kappa$ is as in \eqref{def : definition of K and kappa}. For any parabolic Lebesgue point $w\in \Psi(\Lambda,r_1)$ of $H(\cdot,|Du|)$, we choose $\lambda_w>0$ such that
\begin{equation}\label{cond : Lambda=H(lambda)}
\Lambda=\lambda_w^p+a(w)\lambda_w^q=H_w(\lambda_w),
\end{equation}
where $H_{w}$ denotes the function defined in \eqref{def : definition of H with a fixed center z_0} with $z_0$ replaced by $w$. We note that, since $H_w$ is increasing and
$$
H_w(\lambda_w)=\Lambda>\Lambda_0\geq \lambda_0^p+a(w)\lambda_0^q=H_w(\lambda_0),
$$
we have
\begin{equation}\label{cond : lambda_w geq 1}
    \lambda_w\geq \lambda_0\geq 1,
\end{equation}
so that Lemma \ref{lem : oscillation on intrinsic cylinders} is available on every intrinsic cylinder with scaling parameter $\lambda_w$. All super-level sets in the stopping-time and covering arguments may be restricted to these Lebesgue points, since the omitted set has measure zero. Put $d=r_2-r_1$. The choice of $\Lambda$ and the inequality
$$
H_w(t\lambda_0)\leq t^q\Lambda_0\qquad(t\geq1)
$$
give $\lambda_w>(4\kappa r/d)^{(n+2)/2}\lambda_0$. Since $Q_s^{\lambda_w}(w)\subset Q_{2r}(z_0)$ for $s\leq d$, for $d/(2\kappa)\leq s\leq d$ we have
$$
\miint{Q_s^{\lambda_w}(w)}H(z,|Du|)\,dz
\leq \lambda_w^{p-2}\left(\frac{2r}{s}\right)^{n+2}\lambda_0^2
<\lambda_w^p.
$$
At a Lebesgue point $w\in\Psi(\Lambda,r_1)$ the small-radius limit is larger than $\Lambda\geq\lambda_w^p$. Continuity of the integral average in the radius therefore gives a largest crossing radius $\varrho_w\in(0,d/(2\kappa))$ such that
\begin{equation}\label{cond : integral of H in =varrho}
\miint{Q_{\varrho_w}^{\lambda_w}(w)} H(z,|Du|)\, dz =\lambda_w^p
\end{equation}
and
\begin{equation}\label{cond : integral of H in >varrho}
\miint{Q_{\varrho}^{\lambda_w}(w)} H(z,|Du|)\, dz <\lambda_w^p
\end{equation}
for any $\varrho\in(\varrho_w,r_2-r_1)$. Moreover, we obtain from \cite[Subsection 5.1]{2023_Gradient_Higher_Integrability_for_Degenerate_Parabolic_Double-Phase_Systems} that
\begin{equation}\label{cond : relation of lambda_xi and lambda_0}
    \lambda_w\leq \left(\frac{2r}{\varrho_w}\right)^\frac{n+2}{2}\lambda_0.
\end{equation}

For $K>1$ as in \eqref{def : definition of K and kappa}, we consider the following three cases:
\begin{enumerate}
    \item\label{case : p-phase} $\displaystyle K\lambda_w^p\geq \sup_{Q_{10\varrho_w}(w)}a(\cdot)\lambda_w^q$,
    \item\label{case : p,q-phase} $\displaystyle K\lambda_w^p\leq \sup_{Q_{10\varrho_w}(w)}a(\cdot)\lambda_w^q\qquad$ and $\qquad\displaystyle \sup_{Q_{10\varrho_w}(w)}a(\cdot)\geq 4[a]_\alpha (10\varrho_w)^\alpha$,
    \item\label{case : no occurence} $\displaystyle K\lambda_w^p\leq \sup_{Q_{10\varrho_w}(w)}a(\cdot)\lambda_w^q\qquad$ and $\qquad\displaystyle \sup_{Q_{10\varrho_w}(w)}a(\cdot)\leq 4[a]_\alpha (10\varrho_w)^\alpha$.
\end{enumerate}
If the first two cases overlap, we assign $w$ to the $p$-phase.
\textbf{Case \eqref{case : p-phase}}: By using \eqref{cond : integral of H in =varrho} and \eqref{cond : integral of H in >varrho} and replacing the center point $w$, radius $\varrho_w$ and $\lambda_w$ with $z_0$, $\rho$ and $\lambda$, respectively, we obtain
\begin{equation}\label{cond : p-phase condition}
    \left\{\begin{aligned}
        &K\lambda^p\geq \sup_{Q_{10\rho}(z_0)} a(\cdot)\lambda^q,\\
        &\miint{Q_\sigma^\lambda (z_0)} H(z,|Du|)\, dz <\lambda^p \quad \text{for any }\sigma\in(\rho,2\kappa\rho],\\
        &\miint{Q_\rho^\lambda(z_0)} H(z,|Du|)\, dz = \lambda^p.
    \end{aligned}\right.
\end{equation}
\textbf{Case \eqref{case : p,q-phase}}: We obtain from \eqref{case : p,q-phase}$_2$ that
$$
4[a]_\alpha (10\varrho_w)^\alpha\leq \sup_{Q_{10\varrho_w}(w)} a(\cdot)\leq \inf_{Q_{10\varrho_w}(w)}a(\cdot)+2[a]_\alpha (10\varrho_w)^\alpha,
$$
and hence
$$
\sup_{Q_{10\varrho_w}(w)}a(\cdot)\leq \inf_{Q_{10\varrho_w} (w)} a(\cdot) +2[a]_\alpha (10\varrho_w)^\alpha\leq 2\inf_{Q_{10\varrho_w} (w)} a(\cdot).
$$
Therefore, we get
\begin{equation}    \label{cond : comparison in p,q-phase}
    \frac{a(w)}{2}\leq a(\tilde{w})\leq 2a(w) \quad \text{for every } \tilde{w}\in Q_{10\varrho_w} (w).
\end{equation}
It follows from \eqref{case : p,q-phase} and \eqref{cond : comparison in p,q-phase} that $a(w)>0$ and $G_\sigma^{\lambda_w}(w) \varsubsetneq Q_\sigma^{\lambda_w}(w)$. Moreover, $G_s^{\lambda_w}(w)\subset Q_s^{\lambda_w}(w)$ and
$$
\frac{|Q_s^{\lambda_w}(w)|}{|G_s^{\lambda_w}(w)|}=\frac{H_w(\lambda_w)}{\lambda_w^p}.
$$
Thus the average over $G_s^{\lambda_w}(w)$ is at most $H_w(\lambda_w)$ at $s=\varrho_w$ and is strictly smaller for $s\in(\varrho_w,r_2-r_1)$. Taking the largest crossing radius for the $G$-cylinders yields $\varsigma_w\in(0,\varrho_w]$ such that
\begin{equation}\label{cond : integral of H in =varsigma in p,q-phase}
    \miint{G^{\lambda_w}_{\varsigma_w}(w)} H(z,|Du|)\, dz=H_w(\lambda_w)
\end{equation}
and
\begin{equation}\label{cond : integral of H in >varsigma in p,q-phase}
    \miint{G^{\lambda_w}_{\sigma}(w)} H(z,|Du|)\, dz<H_w(\lambda_w)
\end{equation}
for any $\sigma\in(\varsigma_w,r_2-r_1)$. The original phase condition implies
$$
\frac K2\lambda_w^p\leq a(w)\lambda_w^q.
$$
This condition at the center survives the reduction from $\varrho_w$ to $\varsigma_w$, whereas a condition involving the supremum on the smaller cylinder need not do so. Replacing $w$, $\varsigma_w$ and $\lambda_w$ by $z_0$, $\rho$ and $\lambda$, respectively, we therefore obtain
\begin{equation}\label{cond : p,q-phase condition}
    \left\{\begin{aligned}
        &\frac K2\lambda^p\leq a(z_0)\lambda^q,\quad \frac{a(z_0)}{2}\leq a(z)\leq 2a(z_0)\quad \text{for every }z\in Q_{10\rho}(z_0),\\
        &\miint{G_\sigma^\lambda (z_0)} H(z,|Du|)\, dz <H_{z_0}(\lambda) \quad \text{for any }\sigma\in(\rho,2\kappa\rho],\\
        &\miint{G_\rho^\lambda(z_0)} H(z,|Du|)\, dz = H_{z_0}(\lambda).
    \end{aligned}\right.
\end{equation}
\textbf{Case \eqref{case : no occurence}}: We shall rigorously exclude the possibility of this case by proving the estimate
\begin{equation}\label{cond : the impossibility of Case (3)}
    \lambda_w \lesssim \varrho_w^{-(1-\gamma)},
\end{equation}
which improves the corresponding estimate $\lambda_w\lesssim \varrho_w^{-1}$ for bounded solutions in \cite{KimOh2026}.
\begin{lemma}\label{lem : no occurence with gamma}
    Let $u$ be a weak solution to \eqref{eq : main equation}, and suppose that
\begin{equation}\label{cond : no occurence}
    \sup_{Q_{10\varrho_w}(w)}a(\cdot)\leq 4[a]_\alpha (10\varrho_w)^\alpha.
\end{equation}
    If \eqref{cond : main assumption with gamma} holds, then there exists a constant $c_\gamma=c_\gamma(\operatorname{data}_\gamma)>1$ such that
    $$
    \varrho_w \leq c_\gamma\lambda_w^{-\frac{1}{1-\gamma}}.
    $$
\end{lemma}
\begin{proof}
    By Lemma \ref{lem : Caccioppoli inequality} and \eqref{cond : integral of H in =varrho}, we get
    \begin{align}
    \lambda_w^p&=\miint{Q_{\varrho_w}^{\lambda_w}(w)}H(z,|Du|)\,dz \nonumber\\
    &\leq c\miint{Q_{2\varrho_w}^{\lambda_w}(w)}\left[\frac{\Big|u-u_{Q_{2\varrho_w}^{\lambda_w}(w)}\Big|^p}{(2\varrho_w)^p}+a(z)\frac{\Big|u-u_{Q_{2\varrho_w}^{\lambda_w}(w)}\Big|^q}{(2\varrho_w)^q}\right] dz \nonumber\\
    &\qquad +c\lambda_w^{p-2}\miint{Q_{2\varrho_w}^{\lambda_w}(w)}\frac{\Big|u-u_{Q_{2\varrho_w}^{\lambda_w}(w)}\Big|^2}{(2\varrho_w)^2}\,dz \nonumber\\ \label{eq : Caccioppoli ineq in no occurence with gamma}
    &=\mathrm{I}_1+\mathrm{I}_2+\mathrm{I}_3
    \end{align}
    for some $c=c(n,p,q,\nu,L)>1$. By \eqref{cond : lambda_w geq 1} and Lemma \ref{lem : oscillation on intrinsic cylinders}, we have
    \begin{equation}\label{eq : oscillation estimate in no occurence}
        \Big|u(z)-u_{Q_{2\varrho_w}^{\lambda_w}(w)}\Big|\leq 2[u]_\gamma (2\varrho_w)^\gamma \qquad \text{for every } z\in Q^{\lambda_w}_{2\varrho_w}(w).
    \end{equation}

    We first estimate $\mathrm{I}_1$. By \eqref{eq : oscillation estimate in no occurence}, we get
    $$
    \mathrm{I}_1\leq c[u]_\gamma^p (2\varrho_w)^{(\gamma-1)p}\leq c\varrho_w^{-(1-\gamma)p}
    $$
    for some $c=c(n,p,q,\nu,L,[u]_\gamma)>1$.

    Next, we estimate $\mathrm{I}_2$. Since $\lambda_w\geq 1$ and $p\geq 2$, we have $Q^{\lambda_w}_{2\varrho_w}(w)\subset Q_{10\varrho_w}(w)$. Hence, by \eqref{cond : no occurence}, \eqref{eq : oscillation estimate in no occurence} and \eqref{cond : main assumption with gamma}$_3$, we get
    $$
    \mathrm{I}_2\leq c\varrho_w^\alpha [u]_\gamma^q (2\varrho_w)^{(\gamma-1)q}\leq c\varrho_w^{\alpha-(1-\gamma)q}=c\varrho_w^{\alpha-(1-\gamma)(q-p)}\varrho_w^{-(1-\gamma)p}\leq c\varrho_w^{-(1-\gamma)p}
    $$
    for some $c=c(n,p,q,\alpha,\gamma,\nu,L,\operatorname{diam}(\Omega),[a]_\alpha,[u]_\gamma)>1$, where in the last step we used that $\alpha-(1-\gamma)(q-p)\geq 0$ by \eqref{cond : main assumption with gamma}$_3$ and that $\varrho_w\leq \operatorname{diam}(\Omega)$.

    Finally, we estimate $\mathrm{I}_3$. By \eqref{eq : oscillation estimate in no occurence}, we get
    $$
    \mathrm{I}_3\leq c\lambda_w^{p-2}[u]_\gamma^2 (2\varrho_w)^{2(\gamma-1)}\leq c \lambda_w^{p-2}\varrho_w^{-2(1-\gamma)}.
    $$
    If $p=2$, this already gives $\mathrm{I}_3\leq c\varrho_w^{-(1-\gamma)p}$. If $p>2$, then Young's inequality with exponents $\frac{p}{p-2}$ and $\frac{p}{2}$ yields
    $$
    \mathrm{I}_3\leq \frac{1}{2}\lambda_w^p+c\varrho_w^{-(1-\gamma)p}
    $$
    for some $c=c(n,p,q,\nu,L,[u]_\gamma)>1$.

    Combining the above results with \eqref{eq : Caccioppoli ineq in no occurence with gamma}, we conclude that
    $$
    \lambda_w^p\leq c\varrho_w^{-(1-\gamma)p},
    $$
    and the claim follows by taking the $p$-th root and rearranging.
\end{proof}

Now, we show that the case \eqref{case : no occurence} never occurs. If \eqref{case : no occurence} holds, we have
$$
K\lambda_w^p=\sup_{Q_{10\varrho_w}(w)} a(\cdot)\frac{K\lambda_w^p}{\displaystyle\sup_{Q_{10\varrho_w}(w)} a(\cdot)}\leq 40[a]_\alpha\varrho_w^\alpha \lambda_w^q.
$$
It then follows from Lemma \ref{lem : no occurence with gamma}, \eqref{cond : lambda_w geq 1}, \eqref{cond : main assumption with gamma}$_3$ and \eqref{def : definition of K and kappa} that
$$
K\lambda_w^p\leq 40[a]_\alpha\varrho_w^\alpha \lambda_w^q\leq 40c_\gamma^\alpha[a]_\alpha\lambda_w^{q-\frac{\alpha}{1-\gamma}}\leq 40c_\gamma[a]_\alpha\lambda_w^{p}<\frac{K}{2}\lambda_w^p,
$$
which is a contradiction. Thus, the case \eqref{case : no occurence} can never happen under \eqref{cond : main assumption with gamma}.

\section{\bf Parabolic Sobolev-Poincar\'{e} type inequalities}\label{section 4}
Let $z_0=(x_0,t_0)\in\Psi(\Lambda)$ be a Lebesgue point of $|Du(z)|^p+a(z)|Du(z)|^q$, where $\Lambda$ is defined in Section \ref{section 3}. In this section we collect the parabolic Sobolev-Poincar\'{e} type inequalities for the $p$-phase and the $(p,q)$-phase. We emphasize that none of the statements in this section involves the assumption \eqref{cond : main assumption with gamma}$_1$ on $u$. They depend on $u$ only through the phase conditions \eqref{cond : p-phase condition} and \eqref{cond : p,q-phase condition}. We follow \cite[Section 4]{KimOh2026}, which in turn follows \cite{2023_Gradient_Higher_Integrability_for_Degenerate_Parabolic_Double-Phase_Systems}. We also note that the exponent range $\theta\in(\frac{q-1}{p},1]$ appearing in the $(p,q)$-phase case is nonempty precisely because $q<p+1$, see \eqref{cond : main assumption with gamma}$_4$ and Remark \ref{rem : role of q<p+1}, while the estimates in the $p$-phase case hold for the wider range $\theta\in(\frac{q-1}{q},1]$.

\subsection{\bf The $p$-phase case}
We assume \eqref{cond : p-phase condition} and $\lambda\geq1$. Whenever $\Lambda$ is used below, it denotes $H_{z_0}(\lambda)$, as in the stopping-time construction.
\begin{lemma}   \label{lem : last term estimate in p-intrinsic cylinder}
    Let $u$ be a weak solution to \eqref{eq : main equation} and assume that $Q_{4\rho}^\lambda(z_0)\subset \Omega_T$ satisfies \eqref{cond : p-phase condition}. Then, for $\sigma\in[2\rho,4\rho]$, there exists a positive constant $c=c(\operatorname{data}_\gamma)$ such that
    $$
    \begin{aligned}
        \miint{Q_\sigma^\lambda(z_0)} \left[|Du|^{p-1}+a(z)|Du|^{q-1}\right] dz&\leq \miint{Q_\sigma^\lambda(z_0)} |Du|^{p-1}\, dz\\
        &\quad +c\lambda^{-1+\frac{p}{q}}\miint{Q_\sigma^\lambda(z_0)} a(z)^{\frac{q-1}{q}}|Du|^{q-1}\, dz.
    \end{aligned}
    $$
\end{lemma}
\begin{proof}
    By \eqref{cond : p-phase condition}$_1$, there exists a positive constant $c=c(K)$ such that
    \begin{align*}
        &\miint{Q_\sigma^\lambda(z_0)} \left[|Du|^{p-1}+a(z)|Du|^{q-1}\right] dz\\
        &\quad\quad\leq \miint{Q_\sigma^\lambda(z_0)} |Du|^{p-1}\, dz+\sup_{Q_{10\rho}(z_0)} a(\cdot)^{\frac{1}{q}}\miint{Q_\sigma^\lambda(z_0)} a(z)^{\frac{q-1}{q}}|Du|^{q-1}\, dz\\
        &\quad\quad\leq \miint{Q_\sigma^\lambda(z_0)} |Du|^{p-1}\, dz+ c\lambda^{-1+\frac{p}{q}}\miint{Q_\sigma^\lambda(z_0)} a(z)^{\frac{q-1}{q}}|Du|^{q-1}\, dz.
    \end{align*}
\end{proof}
Next, we establish a $p$-intrinsic parabolic Poincar\'{e} inequality.
\begin{lemma}\label{lem : p-intrinsic parabolic Poincare inequality of p-term in p-intrinsic cylinder}
    Let $u$ be a weak solution to \eqref{eq : main equation} and assume that $Q_{4\rho}^\lambda(z_0)\subset \Omega_T$ satisfies \eqref{cond : p-phase condition}. Then, for $\sigma\in[2\rho,4\rho]$ and $\theta \in (\frac{q-1}{q},1]$, there exists a positive constant $c=c(\operatorname{data}_\gamma)$ such that
    \begin{align*}
        \miint{Q_\sigma^\lambda(z_0)}\frac{\Big|u-u_{Q_\sigma^\lambda(z_0)}\Big|^{\theta p}}{\sigma^{\theta p}}\, dz\leq c\miint{Q_\sigma^\lambda(z_0)} H(z,|Du|)^\theta\, dz.
    \end{align*}
\end{lemma}
\begin{proof}
    By Lemmas \ref{lem : semi-Parabolic Poincare inequality} and \ref{lem : last term estimate in p-intrinsic cylinder}, there exists a positive constant $c=c(\operatorname{data}_\gamma)$ such that
    \begin{align*}
        \miint{Q_\sigma^\lambda(z_0)}\frac{\Big|u-u_{Q_\sigma^\lambda(z_0)}\Big|^{\theta p}}{\sigma^{\theta p}}\, dz &\leq  c \miint{Q_\sigma^\lambda(z_0)}|D u|^{\theta p} \,dz\\
        &\quad +c\left(\lambda^{2-p}\miint{Q_\sigma^\lambda(z_0)} |Du|^{p-1}\, dz\right)^{\theta p}\\
        &\quad +c\left(\lambda^{1-p+\frac{p}{q}}\miint{Q_\sigma^\lambda(z_0)} a(z)^{\frac{q-1}{q}}|Du|^{q-1}\, dz\right)^{\theta p}.
    \end{align*}
    We conclude from H\"{o}lder's inequality and \eqref{cond : p-phase condition} that
    \begin{align*}
        &\miint{Q_\sigma^\lambda(z_0)}\frac{\Big|u-u_{Q_\sigma^\lambda(z_0)}\Big|^{\theta p}}{\sigma^{\theta p}}\, dz\\
        &\quad\leq  c \miint{Q_\sigma^\lambda(z_0)}|D u|^{\theta p} \,dz\\
        &\qquad +c\lambda^{(2-p)\theta p}\left(\miint{Q_\sigma^\lambda(z_0)} |Du|^{\theta p}\, dz\right)^{p-1}\\
        &\qquad +c\lambda^{\left(1-p+\frac{p}{q}\right)\theta p}\left(\miint{Q_\sigma^\lambda(z_0)} a(z)^\theta|Du|^{\theta q}\, dz\right)^{p-\frac{p}{q}}\\
        &\quad\leq  c \miint{Q_\sigma^\lambda(z_0)}|D u|^{\theta p} \,dz\\
        &\qquad +c\lambda^{(2-p)\theta p}\left(\miint{Q_\sigma^\lambda(z_0)} |Du|^{\theta p}\, dz\right)\left(\miint{Q_\sigma^\lambda(z_0)} |Du|^{p}\, dz\right)^{(p-2)\theta}\\
        &\qquad +c\lambda^{\left(1-p+\frac{p}{q}\right)\theta p}\left(\miint{Q_\sigma^\lambda(z_0)} a(z)^\theta|Du|^{\theta q}\, dz\right)\left(\miint{Q_\sigma^\lambda(z_0)} a(z)|Du|^{q}\, dz\right)^{\left(-1+p-\frac{p}{q}\right)\theta}\\
        &\quad\leq  c \miint{Q_\sigma^\lambda(z_0)}H(z,|Du|)^{\theta} \,dz.
    \end{align*}
\end{proof}
\begin{lemma}\label{lem : p-intrinsic parabolic Poincare inequality of q-term in p-intrinsic cylinder}
    Let $u$ be a weak solution to \eqref{eq : main equation} and assume that $Q_{4\rho}^\lambda(z_0)\subset \Omega_T$ satisfies \eqref{cond : p-phase condition}. Then, for $\sigma\in[2\rho,4\rho]$ and $\theta \in (\frac{q-1}{q},1]$, there exists a positive constant $c=c(\operatorname{data}_\gamma)$ such that
    \begin{align*}
        \miint{Q_\sigma^\lambda(z_0)}\inf_{w\in Q_\sigma^\lambda(z_0)}a(w)^\theta\frac{\Big|u-u_{Q_\sigma^\lambda(z_0)}\Big|^{\theta q}}{\sigma^{\theta q}}\, dz\leq c\miint{Q_\sigma^\lambda(z_0)} H(z,|Du|)^\theta\, dz.
    \end{align*}
\end{lemma}
\begin{proof}
    By Lemmas \ref{lem : semi-Parabolic Poincare inequality} and \ref{lem : last term estimate in p-intrinsic cylinder}, there exists a positive constant $c=c(\operatorname{data}_\gamma)$ such that
    \begin{align*}
        &\miint{Q_\sigma^\lambda(z_0)}\inf_{w\in Q_\sigma^\lambda(z_0)}a(w)^\theta\frac{\Big|u-u_{Q_\sigma^\lambda(z_0)}\Big|^{\theta q}}{\sigma^{\theta q}}\, dz \\
        &\quad\leq  c \miint{Q_\sigma^\lambda(z_0)} \inf_{w\in Q_\sigma^\lambda(z_0)}a(w)^\theta|D u|^{\theta q} \,dz\\
        &\qquad +c\inf_{w\in Q_\sigma^\lambda(z_0)}a(w)^\theta\left(\lambda^{2-p}\miint{Q_\sigma^\lambda(z_0)} |Du|^{p-1}\, dz\right)^{\theta q}\\
        &\qquad +c\inf_{w\in Q_\sigma^\lambda(z_0)}a(w)^\theta\left(\lambda^{1-p+\frac{p}{q}}\miint{Q_\sigma^\lambda(z_0)} a(z)^{\frac{q-1}{q}}|Du|^{q-1}\, dz\right)^{\theta q}.
    \end{align*}
    By H\"{o}lder's inequality and \eqref{cond : p-phase condition}, we obtain
    \begin{align*}
        &\miint{Q_\sigma^\lambda(z_0)}\inf_{w\in Q_\sigma^\lambda(z_0)}a(w)^\theta\frac{\Big|u-u_{Q_\sigma^\lambda(z_0)}\Big|^{\theta q}}{\sigma^{\theta q}}\, dz\\
        &\quad\leq  c \miint{Q_\sigma^\lambda(z_0)} a(z)^\theta|D u|^{\theta q} \,dz\\
        &\qquad +c\lambda^{(p-q)\theta}\lambda^{(2-p)\theta q}\left(\miint{Q_\sigma^\lambda(z_0)} |Du|^{\theta p}\, dz\right)^{q-\frac{q}{p}}\\
        &\qquad +c\lambda^{(p-q)\theta}\lambda^{\left(1-p+\frac{p}{q}\right)\theta q}\left(\miint{Q_\sigma^\lambda(z_0)} a(z)^\theta|Du|^{\theta q}\, dz\right)^{q-1}\\
        &\quad \leq c \miint{Q_\sigma^\lambda(z_0)} H(z,|Du|)^\theta \,dz.
    \end{align*}
\end{proof}

\subsection{\bf The $(p,q)$-phase case}
The coefficient comparison and the energy bound in \eqref{cond : p,q-phase condition}, together with $K\geq1$, imply for every $\sigma\in(\rho,4\rho]$ that
$$
\miint{G_\sigma^\lambda(z_0)}H_{z_0}(|Du|)\,dz
<2H_{z_0}(\lambda)
\leq2\left(1+\frac2K\right)a(z_0)\lambda^q
\leq6a(z_0)\lambda^q.
$$
Consequently,
\begin{equation}\label{cond : simple p,q-phase condition}
    \miint{G_\sigma^\lambda(z_0)}|Du|^q\,dz<6\lambda^q
    \qquad\text{for every }\sigma\in(\rho,4\rho].
\end{equation}

We now establish a $(p,q)$-intrinsic parabolic Poincar\'{e} inequality.
\begin{lemma}\label{lem : p,q-intrinsic parabolic Poincare inequality}
    Let $u$ be a weak solution to \eqref{eq : main equation} and assume that $G_{4\rho}^\lambda(z_0)\subset \Omega_T$ satisfies \eqref{cond : p,q-phase condition}. Then, for $\sigma\in[2\rho,4\rho]$ and $\theta \in (\frac{q-1}{p},1]$, there exists a positive constant $c=c(n,p,q,L)$ such that
    $$
    \miint{G_\sigma^\lambda (z_0)} H^\theta_{z_0}\left(\frac{\Big|u-u_{G_\sigma^\lambda(z_0)}\Big|}{\sigma}\right)\, dz\leq c \miint{G_\sigma^\lambda(z_0)} H_{z_0}^\theta(|Du|)\, dz.
    $$
\end{lemma}
\begin{proof}
    We only sketch the proof. For a detailed proof, see \cite[Lemma 3.4]{2023_Gradient_Higher_Integrability_for_Degenerate_Parabolic_Double-Phase_Systems}.

    By Lemma \ref{lem : semi-Parabolic Poincare inequality}, \eqref{cond : p,q-phase condition} and \eqref{cond : simple p,q-phase condition}, there exists a positive constant $c=c(n,p,q,L)$ such that
    $$
    \begin{aligned}
        &\miint{G_\sigma^\lambda(z_0)} H_{z_0}^\theta\left(\frac{\Big|u-u_{G_\sigma^\lambda(z_0)}\Big|}{\sigma}\right)dz\\
        &\quad \leq c\miint{G_\sigma^\lambda(z_0)} H_{z_0}^\theta(|Du|)\, dz+cH^\theta_{z_0}\left(\frac{\lambda}{H'_{z_0}(\lambda)}\miint{G_\sigma^\lambda(z_0)} H'_{z_0}(|Du|)\,dz\right)\\
        &\quad \leq c\miint{G_\sigma^\lambda(z_0)} H_{z_0}^\theta(|Du|)\, dz+cH^\theta_{z_0}\left(\lambda^{2-p}\left(\miint{G_\sigma^\lambda(z_0)}|Du|^{q-1}\, dz\right)^{\frac{p-1}{q-1}}\right).
    \end{aligned}
    $$
    To justify the second inequality, put
    $$
    M=\left(\miint{G_\sigma^\lambda(z_0)}|Du|^{q-1}\,dz\right)^{1/(q-1)}.
    $$
    By \eqref{cond : simple p,q-phase condition}, $M\leq c\lambda$. Thus
    $$
    \frac{\lambda}{H'_{z_0}(\lambda)}\miint{G_\sigma^\lambda(z_0)}H'_{z_0}(|Du|)\,dz
    \leq c\lambda^{2-p}M^{p-1}.
    $$
    For $m=p,q$, write $A_m=\miint{G_\sigma^\lambda(z_0)}|Du|^{\theta m}\,dz$. Since $\theta m>q-1$, H\"older's inequality gives
    $$
    \left(\lambda^{2-p}M^{p-1}\right)^{\theta m}
    \leq \lambda^{(2-p)\theta m}A_m^{p-1}
    \leq c A_m,
    $$
    where $A_m\leq c\lambda^{\theta m}$ and $p\geq2$ were used in the last step. This also covers $p=2$, when the factor $A_m^{p-2}$ is absent. In particular, we obtain
    $$
    \begin{aligned}
        &H^\theta_{z_0}\left(\lambda^{2-p}\left(\miint{G_\sigma^\lambda(z_0)}|Du|^{q-1}\, dz\right)^{\frac{p-1}{q-1}}\right)\\
        &\quad\leq c\left(\lambda^{(2-p)p}\left(\miint{G_\sigma^\lambda(z_0)}|Du|^{q-1}\, dz\right)^{\frac{p(p-1)}{q-1}}\right)^\theta\\
        &\qquad + c\left(a(z_0)\lambda^{(2-p)q}\left(\miint{G_\sigma^\lambda(z_0)}|Du|^{q-1}\, dz\right)^{\frac{q(p-1)}{q-1}}\right)^\theta\\
        &\quad \leq c\miint{G_\sigma^\lambda(z_0)}H^\theta_{z_0}(|Du|)\, dz,
    \end{aligned}
    $$
    and the lemma follows.
\end{proof}
According to \cite[Lemma 3.5]{2023_Gradient_Higher_Integrability_for_Degenerate_Parabolic_Double-Phase_Systems}, if we replace $H_{z_0}^\theta(\kappa)$ with $\kappa^{\theta p}$, we get the following result.
\begin{lemma}\label{lem : p,q-intrinsic parabolic Poincare inequality of p-term}
    Let $u$ be a weak solution to \eqref{eq : main equation} and assume that $G_{4\rho}^\lambda(z_0)\subset \Omega_T$ satisfies \eqref{cond : p,q-phase condition}. Then, for $\sigma\in[2\rho,4\rho]$ and $\theta \in (\frac{q-1}{p},1]$, there exists a positive constant $c=c(n,p,q,L)$ such that
    $$
    \miint{G_\sigma^\lambda(z_0)}\left(\frac{\Big|u-u_{G_\sigma^\lambda(z_0)}\Big|}{\sigma}\right)^{\theta p}\, dz\leq c\miint{G_\sigma^\lambda(z_0)} |Du|^{\theta p}\, dz.
    $$
\end{lemma}

\section{\bf Reverse H\"{o}lder inequalities}\label{section 5}
In this section, we establish reverse H\"{o}lder inequalities separately in each intrinsic cylinder. For this, we need the following auxiliary lemmas, called the Gagliardo-Nirenberg inequality and a standard iteration lemma, see \cite[Lemma 2.12]{Hasto_2021} and \cite[Lemma 6.1]{2003_Giusti_Direct_methods_in_the_calculus_of_variations}, respectively.
\begin{lemma}   \label{lem : Gagliardo-Nirenberg inequality}
    For an open ball $B=B_{\rho}(x_0)\subset \mr^n$, take $p_1,\,p_2,\,p_3\in[1,\infty)$, $\vartheta\in(0,1)$ and let $\psi\in W^{1,p_2}(B)$. Suppose that
    $$
    -\frac{n}{p_1}\leq \vartheta\left(1-\frac{n}{p_2}\right)-(1-\vartheta)\frac{n}{p_3}.
    $$
    Then there exists a positive constant $c=c(n,p_1)$ such that
    $$
    \dashint_B \frac{|\psi|^{p_1}}{\rho^{p_1}}\, dx\leq c \left(\dashint_{B_\rho(x_0)}\left[\frac{|\psi|^{p_2}}{\rho^{p_2}}+|D\psi|^{p_2}\right]dx\right)^{\frac{\vartheta p_1}{p_2}}\left(\dashint_{B_{\rho}(x_0)}\frac{|\psi|^{p_3}}{\rho^{p_3}}\,dx\right)^{\frac{(1-\vartheta)p_1}{p_3}}.
    $$
\end{lemma}
\begin{lemma}   \label{lem : a standard iteration lemma}
    Let $0<\rho<\tau<\infty$, and let $g:[\rho,\tau]\rightarrow [0,\infty)$ be a bounded function. Suppose that
    $$
    g(\rho_1)\leq \vartheta g(\rho_2)+\frac{A}{(\rho_2-\rho_1)^\beta}+B
    $$
    holds for all $0<\rho\leq \rho_1<\rho_2\leq \tau$, where $\vartheta \in (0,1)$, $A,B\geq 0$ and $\beta>0$. Then there exists a positive constant $c$ depending on $\vartheta$ and $\beta$ such that
    $$
    g(\rho)\leq c\left(\frac{A}{(\tau-\rho)^{\beta}}+B\right).
    $$
\end{lemma}

\subsection{The $p$-phase case.}
In this case, we consider the $p$-intrinsic cylinder denoted in \eqref{eq : definition of p-intrinsic cylinder}. We write
$$
S(u,Q_{\rho}^\lambda(z_0))=\sup_{I_\rho^\lambda(t_0)}\dashint_{B_\rho(x_0)}\frac{\Big|u-u_{Q_\rho^\lambda(z_0)}\Big|^2}{\rho^2}\, dx.
$$
Note that, by \eqref{cond : lambda_w geq 1}, every intrinsic cylinder arising from the stopping time argument satisfies $\lambda\geq 1$. We henceforth include this in the standing hypotheses of this section, so that Lemma \ref{lem : oscillation on intrinsic cylinders} is at our disposal.
\begin{lemma}   \label{lem : estimation of S in p-intrinsic cylinder}
    Let $u$ be a weak solution to \eqref{eq : main equation} with \eqref{cond : main assumption with gamma} and let $\lambda\geq 1$. Then there exists a constant $c=c(\operatorname{data}_\gamma)>1$ such that
    $$
    S(u,Q_{2\rho}^\lambda(z_0))\leq c \lambda^2,
    $$
    whenever $Q_{2\kappa\rho}^\lambda(z_0)\subset \Omega_T$ satisfies \eqref{cond : p-phase condition}.
\end{lemma}
\begin{proof}
    Let $2\rho\leq\rho_1<\rho_2\leq 4\rho$. By Lemma \ref{lem : Caccioppoli inequality}, there exists a positive constant $c=c(n,p,q,\nu,L)$ such that
    \begin{align}
        &\lambda^{p-2}S(u,Q_{\rho_1}^\lambda(z_0))\nonumber\\
        &\quad \;\leq \frac{c\rho_2^q}{(\rho_2-\rho_1)^q}\miint{Q_{\rho_2}^\lambda (z_0)}\left[\frac{\Big|u-u_{Q_{\rho_2}^\lambda(z_0)}\Big|^p}{\rho_2^p}+a(z)\frac{\Big|u-u_{Q_{\rho_2}^\lambda(z_0)}\Big|^q}{\rho_2^q}\right] dz\nonumber\\\label{eq : estimation of S in p-intrinsic cylinder}
        &\qquad +\frac{c\rho_2^2\lambda^{p-2}}{(\rho_2-\rho_1)^2}\miint{Q_{\rho_2}^\lambda(z_0)}\frac{\Big|u-u_{Q_{\rho_2}^\lambda(z_0)}\Big|^2}{\rho_2^2}\, dz.
    \end{align}
    To estimate the $p$-term in the first term on the right-hand side of \eqref{eq : estimation of S in p-intrinsic cylinder}, we use Lemma \ref{lem : p-intrinsic parabolic Poincare inequality of p-term in p-intrinsic cylinder} with $\theta=1$ and \eqref{cond : p-phase condition}$_2$. Then we have
    \begin{equation}    \label{eq : estimation u^p in p-intrinsic cylinder}
    \miint{Q_{\rho_2}^\lambda(z_0)}\frac{\Big|u-u_{Q_{\rho_2}^\lambda(z_0)}\Big|^p}{\rho_2^p}\,dz\leq c\lambda^p
    \end{equation}
    for some $c=c(\operatorname{data}_\gamma)>1$. Now, we estimate the $q$-term in the first term. Since $a(\cdot) \in C^{\alpha,\frac{\alpha}{2}}(\Omega_T)$, we get
    $$
    \begin{aligned}
        \miint{Q_{\rho_2}^\lambda(z_0)} a(z)\frac{\Big|u-u_{Q_{\rho_2}^\lambda(z_0)}\Big|^q}{\rho_2^q}\, dz &\leq \miint{Q_{\rho_2}^\lambda(z_0)} \inf_{w\in Q_{\rho_2}^\lambda (z_0)} a(w)\frac{\Big|u-u_{Q_{\rho_2}^\lambda (z_0)}\Big|^q}{\rho_2^q}\, dz\\
        &\quad +2[a]_\alpha \rho_2^\alpha \miint{Q_{\rho_2}^\lambda(z_0)} \frac{\Big|u-u_{Q_{\rho_2}^\lambda (z_0)}\Big|^q}{\rho_2^q}\, dz.
    \end{aligned}
    $$
    By Lemma \ref{lem : p-intrinsic parabolic Poincare inequality of q-term in p-intrinsic cylinder} with $\theta = 1$ and \eqref{cond : p-phase condition}$_2$, we obtain
    $$
    \miint{Q_{\rho_2}^\lambda(z_0)} \inf_{w\in Q_{\rho_2}^\lambda (z_0)} a(w)\frac{\Big|u-u_{Q_{\rho_2}^\lambda (z_0)}\Big|^q}{\rho_2^q}\, dz\leq c \lambda^p
    $$
    for some $c=c(\operatorname{data}_\gamma)>1$. For the remaining term, Lemma \ref{lem : oscillation on intrinsic cylinders}, \eqref{cond : main assumption with gamma}$_3$ and \eqref{eq : estimation u^p in p-intrinsic cylinder} yield
    \begin{equation}\label{eq : estimation u^q in p-intrinsic cylinder}
    \begin{aligned}
        \rho_2^\alpha \miint{Q_{\rho_2}^\lambda(z_0)}\frac{\Big|u-u_{Q_{\rho_2}^\lambda (z_0)}\Big|^q}{\rho_2^q}\, dz &= \rho_2^\alpha \miint{Q_{\rho_2}^\lambda(z_0)}\frac{\Big|u-u_{Q_{\rho_2}^\lambda (z_0)}\Big|^{q-p}}{\rho_2^{q-p}}\cdot\frac{\Big|u-u_{Q_{\rho_2}^\lambda (z_0)}\Big|^{p}}{\rho_2^p}\, dz\\
        &\leq c[u]_\gamma^{q-p}\rho_2^{\alpha-(1-\gamma)(q-p)} \miint{Q_{\rho_2}^\lambda(z_0)}\frac{\Big|u-u_{Q_{\rho_2}^\lambda (z_0)}\Big|^{p}}{\rho_2^p}\, dz\\
        &\leq c\lambda^p
    \end{aligned}
    \end{equation}
    for some $c=c(\operatorname{data}_\gamma)>1$, where we used $\alpha-(1-\gamma)(q-p)\geq 0$ and $\rho_2\leq \operatorname{diam}(\Omega)$. Finally, since $p\geq2$, H\"older's inequality and \eqref{eq : estimation u^p in p-intrinsic cylinder} give
    \begin{equation}\label{eq : estimation u^2 in p-intrinsic cylinder}
    \miint{Q_{\rho_2}^\lambda(z_0)}\frac{\big|u-u_{Q_{\rho_2}^\lambda(z_0)}\big|^2}{\rho_2^2}\,dz
    \leq\left(\miint{Q_{\rho_2}^\lambda(z_0)}\frac{\big|u-u_{Q_{\rho_2}^\lambda(z_0)}\big|^p}{\rho_2^p}\,dz\right)^{2/p}
    \leq c\lambda^2.
    \end{equation}
    Inserting the three moment estimates into \eqref{eq : estimation of S in p-intrinsic cylinder} with $\rho_1=2\rho$ and $\rho_2=4\rho$ yields
    $$
    \lambda^{p-2}S(u,Q_{2\rho}^\lambda(z_0))\leq c\lambda^p,
    $$
    which proves the claim.
\end{proof}
\begin{lemma}\label{lem : estimation of difference of u and mean u in p intrinsic cylinder}
    Let $u$ be a weak solution to \eqref{eq : main equation} with \eqref{cond : main assumption with gamma} and let $\lambda\geq 1$. Then there exist constants $c=c(\operatorname{data}_\gamma)>1$ and $\theta_1=\theta_1(n)\in (0,1)$ such that for any $\theta\in(\theta_1,1)$
    $$
    \begin{aligned}
        &\miint{Q_{2\rho}^\lambda (z_0)}\left[\frac{\Big|u-u_{Q_{2\rho}^\lambda (z_0)}\Big|^p}{(2\rho)^p}+a(z)\frac{\Big|u-u_{Q_{2\rho}^\lambda (z_0)}\Big|^q}{(2\rho)^q}\right] dz\\
        &\quad \leq c\lambda^{(1-\theta)p}\miint{Q_{2\rho}^\lambda (z_0)} \left[\frac{\Big|u-u_{Q_{2\rho}^\lambda (z_0)}\Big|^{\theta p}}{(2\rho)^{\theta p}}+|Du|^{\theta p}\right] dz\\
        &\qquad + c\lambda^{(1-\theta)p}\miint{Q_{2\rho}^\lambda (z_0)}\inf_{w\in Q_{2\rho}^\lambda (z_0)}a(w)^\theta \left[\frac{\Big|u-u_{Q_{2\rho}^\lambda (z_0)}\Big|^{\theta q}}{(2\rho)^{\theta q}}+|Du|^{\theta q}\right] dz,
    \end{aligned}
    $$
    whenever $Q_{2\kappa \rho}^\lambda (z_0)\subset \Omega_T$ satisfies \eqref{cond : p-phase condition}.
\end{lemma}
\begin{proof}
    Since $a(\cdot) \in C^{\alpha,\frac{\alpha}{2}}(\Omega_T)$, we see that
    \begin{align}
        &\miint{Q_{2\rho}^\lambda (z_0)}\left[\frac{\Big|u-u_{Q_{2\rho}^\lambda (z_0)}\Big|^p}{(2\rho)^p}+a(z)\frac{\Big|u-u_{Q_{2\rho}^\lambda (z_0)}\Big|^q}{(2\rho)^q}\right] dz\nonumber\\
        &\quad \leq \miint{Q_{2\rho}^\lambda (z_0)}\frac{\Big|u-u_{Q_{2\rho}^\lambda (z_0)}\Big|^p}{(2\rho)^p}\, dz+ \miint{Q_{2\rho}^\lambda (z_0)} \inf_{w\in Q_{2\rho}^\lambda (z_0)} a(w)\frac{\Big|u-u_{Q_{2\rho}^\lambda (z_0)}\Big|^q}{(2\rho)^q}\, dz\nonumber\\ \label{eq : energy estimate in p-intrinsic cylinder}
        &\qquad +2[a]_\alpha (2\rho)^\alpha \miint{Q_{2\rho}^\lambda (z_0)} \frac{\Big|u-u_{Q_{2\rho}^\lambda (z_0)}\Big|^q}{(2\rho)^q}\, dz.
    \end{align}
    By applying Lemma \ref{lem : Gagliardo-Nirenberg inequality} in the same manner as in \cite[Lemma 4.4]{2023_Gradient_Higher_Integrability_for_Degenerate_Parabolic_Double-Phase_Systems} and using Lemma \ref{lem : estimation of S in p-intrinsic cylinder}, we obtain that for any $\theta\in\left(\frac{n}{n+2},1\right)$,
    \begin{equation}\label{eq : p,q-energy estimate in p-intrinsic cylinder}
        \begin{aligned}
            &\miint{Q_{2\rho}^\lambda (z_0)}\frac{\Big|u-u_{Q_{2\rho}^\lambda (z_0)}\Big|^p}{(2\rho)^p}\, dz+ \miint{Q_{2\rho}^\lambda (z_0)} \inf_{w\in Q_{2\rho}^\lambda (z_0)} a(w)\frac{\Big|u-u_{Q_{2\rho}^\lambda (z_0)}\Big|^q}{(2\rho)^q}\, dz\\
            &\quad \leq c\lambda^{(1-\theta)p}\miint{Q_{2\rho}^\lambda (z_0)}\left[\frac{\Big|u-u_{Q_{2\rho}^\lambda (z_0)}\Big|^{\theta p}}{(2\rho)^{\theta p}}+|Du|^{\theta p}\right] dz\\
            &\qquad + c\lambda^{(1-\theta)p}\miint{Q_{2\rho}^\lambda (z_0)}\inf_{w\in Q_{2\rho}^\lambda (z_0)}a(w)^\theta \left[\frac{\Big|u-u_{Q_{2\rho}^\lambda (z_0)}\Big|^{\theta q}}{(2\rho)^{\theta q}}+|Du|^{\theta q}\right] dz.
        \end{aligned}
    \end{equation}
    Moreover, we obtain from Lemma \ref{lem : oscillation on intrinsic cylinders}, \eqref{cond : main assumption with gamma}$_3$ and \eqref{eq : p,q-energy estimate in p-intrinsic cylinder} that
    $$
    \begin{aligned}
    &(2\rho)^\alpha\miint{Q^\lambda_{2\rho}(z_0)} \frac{\Big|u-u_{Q^\lambda_{2\rho}(z_0)}\Big|^q}{(2\rho)^q}\, dz\\
    &\qquad =(2\rho)^\alpha\miint{Q^\lambda_{2\rho}(z_0)} \frac{\Big|u-u_{Q^\lambda_{2\rho}(z_0)}\Big|^{q-p}}{(2\rho)^{q-p}}\cdot\frac{\Big|u-u_{Q^\lambda_{2\rho}(z_0)}\Big|^p}{(2\rho)^p}\, dz\\
    &\qquad \leq c[u]_\gamma^{q-p}(2\rho)^{\alpha-(1-\gamma)(q-p)} \miint{Q^\lambda_{2\rho}(z_0)}\frac{\Big|u-u_{Q^\lambda_{2\rho}(z_0)}\Big|^p}{(2\rho)^p}\, dz\\
    &\qquad\leq c\lambda^{(1-\theta)p}\miint{Q_{2\rho}^\lambda (z_0)}\left[\frac{\Big|u-u_{Q_{2\rho}^\lambda (z_0)}\Big|^{\theta p}}{(2\rho)^{\theta p}}+|Du|^{\theta p}\right] dz\\
    &\qquad\quad + c\lambda^{(1-\theta)p}\miint{Q_{2\rho}^\lambda (z_0)}\inf_{w\in Q_{2\rho}^\lambda (z_0)}a(w)^\theta \left[\frac{\Big|u-u_{Q_{2\rho}^\lambda (z_0)}\Big|^{\theta q}}{(2\rho)^{\theta q}}+|Du|^{\theta q}\right] dz
    \end{aligned}
    $$
    for some $c=c(\operatorname{data}_\gamma)>1$, where we used $\alpha-(1-\gamma)(q-p)\geq 0$ and $\rho\leq \operatorname{diam}(\Omega)$. Hence, combining the above inequalities with \eqref{eq : energy estimate in p-intrinsic cylinder}, we conclude the desired estimate for any $\theta\in(\theta_1,1)$, where $\theta_1 \coloneq \frac{n}{n+2}\in(0,1)$.
\end{proof}
\begin{lemma}\label{lem : pre reverse Holder inequality for p phase}
    Let $u$ be a weak solution to \eqref{eq : main equation} with \eqref{cond : main assumption with gamma} and let $\lambda\geq 1$. Then there exist constants $c=c(\operatorname{data}_\gamma)>1$ and $\theta_0=\theta_0(n,p,q)\in (0,1)$ such that, for any $\theta\in(\theta_0,1)$,
    $$
    \begin{aligned}
        &\miint{Q_\rho^\lambda(z_0)} H(z,|Du|)\, dz\leq c\left(\miint{Q_{2\rho}^\lambda(z_0)} H(z,|Du|)^\theta\,dz\right)^\frac{1}{\theta},
    \end{aligned}
    $$
    whenever $Q_{2\kappa \rho}^\lambda (z_0)\subset \Omega_T$ satisfies \eqref{cond : p-phase condition}.
\end{lemma}
\begin{proof}
    By Lemma \ref{lem : Caccioppoli inequality}, we get
    $$
    \begin{aligned}
    \miint{Q_\rho^\lambda(z_0)} H(z,|Du|)\, dz &\leq c \miint{Q_{2\rho}^\lambda(z_0)} \left[\frac{\Big|u-u_{Q_{2\rho}^\lambda(z_0)}\Big|^p}{(2\rho)^p}+a(z)\frac{\Big|u-u_{Q_{2\rho}^\lambda(z_0)}\Big|^q}{(2\rho)^q}\right] dz\\
    &\quad +c \lambda^{p-2} \miint{Q_{2\rho}^\lambda (z_0)} \frac{\Big|u-u_{Q_{2\rho}^\lambda(z_0)}\Big|^2}{(2\rho)^2}\, dz
    \end{aligned}
    $$
    for some $c=c(n,p,q,\nu,L)>1$. Using Lemmas \ref{lem : p-intrinsic parabolic Poincare inequality of p-term in p-intrinsic cylinder}, \ref{lem : p-intrinsic parabolic Poincare inequality of q-term in p-intrinsic cylinder}, \ref{lem : estimation of difference of u and mean u in p intrinsic cylinder} and Young's inequality yields
    $$
    \begin{aligned}
        &c\miint{Q_{2\rho}^\lambda(z_0)} \left[\frac{\Big|u-u_{Q_{2\rho}^\lambda(z_0)}\Big|^p}{(2\rho)^p}+a(z)\frac{\Big|u-u_{Q_{2\rho}^\lambda(z_0)}\Big|^q}{(2\rho)^q}\right] dz\\
        &\qquad \leq c_0\lambda^{p(1-\theta)}\miint{Q_{2\rho}^\lambda (z_0)} H(z,|Du|)^\theta \,dz\\
        &\qquad \leq \frac{1}{4}\lambda^p +c\left(\miint{Q_{2\rho}^\lambda (z_0)} H(z,|Du|)^\theta \,dz\right)^\frac{1}{\theta}
    \end{aligned}
    $$
    for any $\theta\in (\theta_2,1)$ and for some $c_0,\,c=c(\operatorname{data}_\gamma)>1$, where $\theta_2\coloneq\max\{\frac{n}{n+2},\frac{q-1}{q}\}$. Next, by \cite[Lemma 4.5]{2023_Gradient_Higher_Integrability_for_Degenerate_Parabolic_Double-Phase_Systems}, we obtain
    $$
    \miint{Q_{2\rho}^\lambda (z_0)} \frac{\Big|u-u_{Q_{2\rho}^\lambda(z_0)}\Big|^2}{(2\rho)^2}\, dz\leq c \lambda\left(\miint{Q_{2\rho}^\lambda(z_0)}\left[\frac{\Big|u-u_{Q_{2\rho}^\lambda(z_0)}\Big|^{\theta p}}{(2\rho)^{\theta p}}+|Du|^{\theta p}\right] dz\right)^\frac{1}{\theta p}
    $$
    for some $c=c(\operatorname{data}_\gamma)>1$ and for any $\theta\in (\frac{2n}{(n+2)p},1)$. It follows from Lemma \ref{lem : p-intrinsic parabolic Poincare inequality of p-term in p-intrinsic cylinder} and Young's inequality that
    $$
    \begin{aligned}
        &c\lambda^{p-2}\miint{Q_{2\rho}^\lambda (z_0)} \frac{\Big|u-u_{Q_{2\rho}^\lambda(z_0)}\Big|^2}{(2\rho)^2}\, dz\\
        &\quad \leq c \lambda^{p-1}\left(\miint{Q_{2\rho}^\lambda(z_0)}\left[\frac{\Big|u-u_{Q_{2\rho}^\lambda(z_0)}\Big|^{\theta p}}{(2\rho)^{\theta p}}+|Du|^{\theta p}\right] dz\right)^\frac{1}{\theta p}\\
        &\quad \leq \frac{1}{4}\lambda^p + c\left(\miint{Q_{2\rho}^\lambda (z_0)} H(z,|Du|)^\theta \,dz\right)^\frac{1}{\theta}.
    \end{aligned}
    $$
    By the third condition in \eqref{cond : p-phase condition}, we conclude that
    $$
    \miint{Q_\rho^\lambda(z_0)} H(z,|Du|)\, dz\leq c\left(\miint{Q_{2\rho}^\lambda (z_0)} H(z,|Du|)^\theta \,dz\right)^\frac{1}{\theta}
    $$
    for any $\theta \in (\theta_0,1)$, where $\theta_0\coloneq\max\{\theta_2,\frac{2n}{(n+2)p}\}=\max\{\frac{n}{n+2},\frac{q-1}{q}\}$.
\end{proof}
Arguing as in \cite[Lemma 4.6]{2023_Gradient_Higher_Integrability_for_Degenerate_Parabolic_Double-Phase_Systems}, we obtain the following consequence.
\begin{lemma}\label{lem : reverse Holder inequality for p phase}
    Let $u$ be a weak solution to \eqref{eq : main equation} with \eqref{cond : main assumption with gamma} and let $\lambda\geq 1$. Then there exist constants $c=c(\operatorname{data}_\gamma)>1$ and $\theta_0=\theta_0(n,p,q)\in(0,1)$ such that, for any $\theta\in(\theta_0,1)$,
    $$
    \iints{Q_{2\kappa\rho}^\lambda (z_0)} H(z,|Du|)\, dz\leq c\Lambda^{1-\theta}\iints{Q_{2\rho}^\lambda (z_0)\cap \Psi(c^{-1}\Lambda)} H(z,|Du|)^\theta \, dz,
    $$
    whenever $Q_{2\kappa \rho}^\lambda (z_0)\subset \Omega_T$ satisfies \eqref{cond : p-phase condition}.
\end{lemma}

\subsection{The $(p,q)$-phase case.}
In this case, we consider the $(p,q)$-intrinsic cylinder denoted in \eqref{eq : definition of p,q-intrinsic cylinder}. We write
$$
S(u,G_\rho^\lambda(z_0))=\sup_{J_\rho^\lambda (t_0)}\dashint_{B_\rho (x_0)}\frac{\Big|u-u_{G_{\rho}^\lambda(z_0)}\Big|^2}{\rho^2}\, dx.
$$
As in \cite{KimOh2026}, the estimates in the $(p,q)$-phase do not involve the assumption on $u$: by \eqref{cond : p,q-phase condition}, the coefficient $a(\cdot)$ is comparable to $a(z_0)$ on $G_{4\rho}^\lambda(z_0)$, so that the full structure function $H_{z_0}$ is available and the following results hold with constants independent of $\|u\|_{C^{\gamma,\frac{\gamma}{2}}(\Omega_T)}$. The estimate integrated over $G_{2\kappa\rho}^\lambda$ additionally contains the dilation factor $\kappa^{n+2}$, which depends on $\operatorname{data}_\gamma$. For the proofs we refer to \cite[Section 4]{2023_Gradient_Higher_Integrability_for_Degenerate_Parabolic_Double-Phase_Systems} and \cite[Section 5]{KimOh2026}.
\begin{lemma}\label{lem : estimation of S in p,q-intrinsic cylinder}
    Let $u$ be a weak solution to \eqref{eq : main equation} with \eqref{cond : main assumption with gamma}. Then there exists a constant $c=c(n,p,q,\nu,L)>1$ such that
    $$
    S(u,G_{2\rho}^\lambda(z_0))\leq c\lambda^2,
    $$
    whenever $G_{2\kappa \rho}^\lambda (z_0)\subset \Omega_T$ satisfies \eqref{cond : p,q-phase condition}.
\end{lemma}
\begin{lemma}\label{lem : estimation of difference of u and mean u in p,q intrinsic cylinder}
    Let $u$ be a weak solution to \eqref{eq : main equation} with \eqref{cond : main assumption with gamma}. Then there exists a constant $c=c(n,p,q)>1$ such that for any $\theta\in(\frac{n}{n+2},1)$,
    $$
    \begin{aligned}
        &\miint{G_{2\rho}^\lambda (z_0)} \left[\frac{\Big|u-u_{G_{2\rho}^\lambda(z_0)}\Big|^p}{(2\rho)^p}+a(z)\frac{\Big|u-u_{G_{2\rho}^\lambda(z_0)}\Big|^q}{(2\rho)^q}\right] dz\\
        &\qquad \leq c\miint{G_{2\rho}^\lambda (z_0)} \left[H_{z_0}^\theta\left(\frac{\Big|u-u_{G_{2\rho}^\lambda(z_0)}\Big|}{2\rho}\right)+H_{z_0}^\theta (|Du|)\right] dz\\
        &\qquad\qquad \times H_{z_0}^{1-\theta}\left(S(u,G_{2\rho}^\lambda (z_0))^\frac{1}{2}\right),
    \end{aligned}
    $$
    whenever $G_{2\kappa \rho}^\lambda (z_0)\subset\Omega_T$ satisfies \eqref{cond : p,q-phase condition}.
\end{lemma}
\begin{lemma}\label{lem : reverse Holder inequality for p,q phase}
    Let $u$ be a weak solution to \eqref{eq : main equation} with \eqref{cond : main assumption with gamma}. Then there exist constants $c=c(n,p,q,\nu,L)>1$ and $\theta_0=\theta_0(n,p,q)\in(0,1)$ such that for any $\theta\in(\theta_0,1)$,
    $$
    \miint{G_\rho^\lambda(z_0)} H_{z_0}(|Du|)\, dz\leq c\left(\miint{G_{2\rho}^\lambda(z_0)}H_{z_0}^\theta(|Du|)\, dz\right)^\frac{1}{\theta},
    $$
    whenever $G_{2\kappa \rho}^\lambda(z_0)\subset \Omega_T$ satisfies \eqref{cond : p,q-phase condition}. One may take $\theta_0=\max\{n/(n+2),(q-1)/p\}<1$. Moreover, with $\Lambda=H_{z_0}(\lambda)$, we have
    $$
    \iints{G_{2\kappa \rho}^\lambda (z_0)} H(z,|Du|)\, dz\leq c\kappa^{n+2}\Lambda^{1-\theta}\iints{G_{2\rho}^\lambda (z_0)\cap \Psi(c^{-1}\Lambda)} H(z,|Du|)^\theta\, dz.
    $$
\end{lemma}

\section{\bf Proof of the main result}\label{section 6}
\subsection{\bf Vitali type covering lemma}\label{subsection 6.1}
To prove a Vitali type covering lemma for the collection of intrinsic cylinders, we write them as
$$
\mathcal{Q}(w)=\left\{
\begin{aligned}
    &Q^{\lambda_w}_{2\varrho_w}(w) &\text{if \eqref{case : p-phase} holds,}\\
    &G^{\lambda_w}_{2\varsigma_w}(w) &\text{if \eqref{case : p,q-phase} holds,}
\end{aligned}\right.
$$
where $\lambda_w,\, \varrho_w$ and $\varsigma_w$ are defined in Section \ref{section 3}. As in \cite[Subsection 5.2]{2023_Gradient_Higher_Integrability_for_Degenerate_Parabolic_Double-Phase_Systems}, we obtain a countable collection $\mathcal{G}$ of pairwise disjoint cylinders in $\mathcal{F}\coloneq \{\mathcal{Q}(w):w\in\Psi(\Lambda,r_1)\}$, where $\mathcal{G}$ satisfies the following two properties:
\begin{itemize}
    \item For each $\mathcal{Q}(z_1)\in\mathcal{F}$, there exists $\mathcal{Q}(z_2)\in\mathcal{G}$ such that
    \begin{equation}\label{cond : intersect intrinsic cylinder}
        \mathcal{Q}(z_1)\cap\mathcal{Q}(z_2)\neq \emptyset
    \end{equation}
    \item For the above points $z_1$ and $z_2$, we get
    \begin{equation}\label{cond : relationship of ell_z_1 and ell_z_2}
        \ell_{z_1}\leq 2\ell_{z_2},
    \end{equation}
    where
    $$
    \ell_{z_i}=\left\{
    \begin{aligned}
        &2\varrho_{z_i} &\text{if \eqref{case : p-phase} holds,}\\
        &2\varsigma_{z_i} &\text{if \eqref{case : p,q-phase} holds.}
    \end{aligned}\right.
    $$
\end{itemize}
Now, we claim that for the above points $z_1$ and $z_2$, there holds
\begin{equation}\label{cond : Vitali covering Q(z_1) and Q(z_2)}
    \mathcal{Q}(z_1)\subset \kappa\mathcal{Q}(z_2).
\end{equation}
We prove the claim by dividing it into the following four cases:
\begin{enumerate}[label=(\roman*)]
    \item\label{case : p-phases} $\mathcal{Q}(z_2)=Q_{\ell_{z_2}}^{\lambda_{z_2}}(z_2)\quad$ and $\quad\mathcal{Q}(z_1)=Q_{\ell_{z_1}}^{\lambda_{z_1}}(z_1)$,
    \item\label{case : p,q-phases} $\mathcal{Q}(z_2)=G_{\ell_{z_2}}^{\lambda_{z_2}}(z_2)\quad$ and $\quad\mathcal{Q}(z_1)=G_{\ell_{z_1}}^{\lambda_{z_1}}({z_1})$,
    \item\label{case : p,q-phase and p-phase} $\mathcal{Q}({z_2})=G_{\ell_{z_2}}^{\lambda_{z_2}}({z_2})\quad$ and $\quad\mathcal{Q}({z_1})=Q_{\ell_{z_1}}^{\lambda_{z_1}}({z_1})$,
    \item\label{case : p-phase and p,q-phase} $\mathcal{Q}({z_2})=Q_{\ell_{z_2}}^{\lambda_{z_2}}({z_2})\quad$ and $\quad\mathcal{Q}({z_1})=G_{\ell_{z_1}}^{\lambda_{z_1}}({z_1})$.
\end{enumerate}
Write $z_i=(x_i,t_i)$. The intersection condition and $\ell_{z_1}\leq2\ell_{z_2}$ imply
$$
|x_1-x_2|<3\ell_{z_2},\qquad |t_1-t_2|<5\ell_{z_2}^2,
$$
because both intrinsic time factors are at most one. In particular, $z_1\in Q_{3\ell_{z_2}}(z_2)$. They also imply that the spatial ball of $\mathcal Q(z_1)$ is contained in $B_{5\ell_{z_2}}(x_2)$.

We compare the parameters using the phase of the selected cylinder $\mathcal Q(z_2)$, whose radius is at least half that of $\mathcal Q(z_1)$. If it is a $p$-cylinder, then $3\ell_{z_2}=6\varrho_{z_2}<10\varrho_{z_2}$ and
$$
a(z_1)\lambda_{z_2}^q\leq K\lambda_{z_2}^p,
\qquad \Lambda=H_{z_2}(\lambda_{z_2})\leq(1+K)\lambda_{z_2}^p.
$$
The second inequality and $\lambda_{z_1}^p\leq\Lambda$ give
$\lambda_{z_1}\leq(1+K)^{1/p}\lambda_{z_2}$. If $\lambda_{z_2}\geq\lambda_{z_1}$, the first inequality gives
$$
\left(\frac{\lambda_{z_2}}{\lambda_{z_1}}\right)^p\Lambda
\leq H_{z_1}(\lambda_{z_2})
\leq(1+K)\lambda_{z_2}^p\leq(1+K)\Lambda,
$$
and hence the reverse comparison follows as well.

If $\mathcal Q(z_2)$ is a $(p,q)$-cylinder, then
$3\ell_{z_2}=6\varsigma_{z_2}<10\varrho_{z_2}$, so \eqref{cond : comparison in p,q-phase} gives
$a(z_2)/2\leq a(z_1)\leq2a(z_2)$. Thus $H_{z_1}$ and $H_{z_2}$ are comparable with factor two. Since both parameters have energy level $\Lambda$ and $H_{z_i}(ts)\geq t^pH_{z_i}(s)$ for $t\geq1$, their ratio is bounded in both directions by $2^{1/p}$.

Both cases imply the convenient common bound
\begin{equation}\label{eq : comparison condition of lambda_z_2 and _z_1}
    (4K)^{-1/p}\lambda_{z_1}\leq\lambda_{z_2}\leq(4K)^{1/p}\lambda_{z_1}.
\end{equation}
Here $\kappa\mathcal Q$ denotes dilation of the spatial radius by $\kappa$ and of the time half-length by $\kappa^2$, with the intrinsic parameter fixed.

To prove the claim, we write $z_1=(x_{1},t_{1})$ and $z_2=(x_{2},t_{2})$ for $x_{1},x_{2}\in\mr^n$ and $t_{1},t_{2}\in\mr$.

\ref{case : p-phases} For any $\tau\in I^{\lambda_{z_1}}_{\ell_{z_1}}(t_{1})$, by \eqref{cond : relationship of ell_z_1 and ell_z_2} and \eqref{eq : comparison condition of lambda_z_2 and _z_1}, we have
$$
\begin{aligned}
    |\tau-t_{2}|&\leq |\tau-t_{1}|+|t_{1}-t_{2}|\leq 2\lambda_{z_1}^{2-p}\ell_{z_1}^2+\lambda_{z_2}^{2-p}\ell_{z_2}^2\\
    &\leq (32K+1) \lambda_{z_2}^{2-p}\ell_{z_2}^2\leq \lambda_{z_2}^{2-p}(\kappa\ell_{z_2})^2.
\end{aligned}
$$
Thus, we get $I^{\lambda_{z_1}}_{\ell_{z_1}}(t_{1})\subset \kappa I^{\lambda_{z_2}}_{\ell_{z_2}}(t_{2})$ and hence $Q_{\ell_{z_1}}^{\lambda_{z_1}}(z_1)\subset \kappa Q_{\ell_{z_2}}^{\lambda_{z_2}}(z_2)$.

\ref{case : p,q-phases} For any $\tau\in J^{\lambda_{z_1}}_{\ell_{z_1}}(t_{1})$, we have
$$
|\tau-t_{2}|\leq |\tau-t_{1}|+|t_{1}-t_{2}|\leq 2\frac{\lambda_{z_1}^2}{H_{z_1}(\lambda_{z_1})}\ell_{z_1}^2+\frac{\lambda_{z_2}^2}{H_{z_2}(\lambda_{z_2})}\ell_{z_2}^2.
$$
We apply \eqref{cond : Lambda=H(lambda)} and \eqref{eq : comparison condition of lambda_z_2 and _z_1} to have
$$
\frac{\lambda_{z_1}^2}{H_{z_1}(\lambda_{z_1})}=\frac{\lambda_{z_1}^2}{\Lambda}\leq 4K\frac{\lambda_{z_2}^2}{\Lambda}=4K\frac{\lambda_{z_2}^2}{H_{z_2}(\lambda_{z_2})}.
$$
By \eqref{cond : relationship of ell_z_1 and ell_z_2}, we obtain
$$
|\tau-t_{2}|\leq (32K+1)\frac{\lambda_{z_2}^2}{H_{z_2}(\lambda_{z_2})}\ell_{z_2}^2\leq \frac{\lambda_{z_2}^2}{H_{z_2}(\lambda_{z_2})} (\kappa \ell_{z_2})^2.
$$
Thus, we have $J^{\lambda_{z_1}}_{\ell_{z_1}}(t_{1})\subset \kappa J^{\lambda_{z_2}}_{\ell_{z_2}}(t_{2})$ and hence $G_{\ell_{z_1}}^{\lambda_{z_1}}(z_1)\subset \kappa G_{\ell_{z_2}}^{\lambda_{z_2}}(z_2)$.

\ref{case : p,q-phase and p-phase} For any $\tau \in I^{\lambda_{z_1}}_{\ell_{z_1}}(t_{1})$, applying \eqref{cond : Lambda=H(lambda)}, we have
$$
|\tau-t_{2}|\leq |\tau-t_{1}|+|t_{1}-t_{2}|\leq 2\lambda_{z_1}^{2-p}\ell_{z_1}^2+\frac{\lambda_{z_2}^2}{H_{z_2}(\lambda_{z_2})}\ell_{z_2}^2=2\lambda_{z_1}^{2-p}\ell_{z_1}^2+\frac{\lambda_{z_2}^2}{\Lambda}\ell_{z_2}^2.
$$
Since $\displaystyle K\lambda_{z_1}^p\geq \sup_{Q_{10\varrho_{z_1}}(z_1)} a(\cdot)\lambda_{z_1}^q\geq a(z_1)\lambda_{z_1}^q$, by \eqref{cond : Lambda=H(lambda)} and \eqref{eq : comparison condition of lambda_z_2 and _z_1}, we get
$$
\lambda_{z_1}^{2-p}\leq \frac{2K\lambda_{z_1}^2}{\lambda_{z_1}^p+a(z_1)\lambda_{z_1}^q}\leq 8K^2\frac{\lambda_{z_2}^2}{\Lambda},
$$
and hence we see from \eqref{cond : relationship of ell_z_1 and ell_z_2} that
$$
|\tau-t_{2}|\leq (64K^2+1)\frac{\lambda_{z_2}^2}{\Lambda}\ell_{z_2}^2\leq \frac{\lambda_{z_2}^2}{\lambda_{z_2}^p+a(z_2)\lambda_{z_2}^q}(\kappa\ell_{z_2})^2.
$$
Thus, we get $I^{\lambda_{z_1}}_{\ell_{z_1}}(t_{1})\subset \kappa J^{\lambda_{z_2}}_{\ell_{z_2}}(t_{2})$ and hence $Q_{\ell_{z_1}}^{\lambda_{z_1}}(z_1)\subset \kappa G_{\ell_{z_2}}^{\lambda_{z_2}}(z_2)$.

\ref{case : p-phase and p,q-phase} For any $\tau\in J^{\lambda_{z_1}}_{\ell_{z_1}}(t_{1})$, using \eqref{cond : Lambda=H(lambda)} and \eqref{eq : comparison condition of lambda_z_2 and _z_1}, we obtain
$$
\begin{aligned}
    |\tau-t_{2}|&\leq |\tau-t_{1}|+|t_{1}-t_{2}|\leq 2\frac{\lambda_{z_1}^2}{H_{z_1}(\lambda_{z_1})}\ell_{z_1}^2+\lambda_{z_2}^{2-p}\ell_{z_2}^2\\
    &\leq 2\lambda_{z_1}^{2-p}\ell_{z_1}^2+\lambda_{z_2}^{2-p}\ell_{z_2}^2\leq (32K+1)\lambda_{z_2}^{2-p}\ell_{z_2}^2\\
    &\leq\lambda_{z_2}^{2-p}(\kappa\ell_{z_2})^2.
\end{aligned}
$$
Therefore, $J^{\lambda_{z_1}}_{\ell_{z_1}}(t_{1})\subset \kappa I^{\lambda_{z_2}}_{\ell_{z_2}}(t_{2})$ and $G_{\ell_{z_1}}^{\lambda_{z_1}}(z_1)\subset \kappa Q_{\ell_{z_2}}^{\lambda_{z_2}}(z_2)$. Thus, the claim is proved.

\subsection{\bf Proof of Theorem \ref{thm : main theorem with gamma}}\label{subsection 6.2}
Denote intrinsic cylinders in the countable pairwise disjoint collection $\mathcal{G}$ by, for $j=1,2,\cdots,$
$$
\mathcal{Q}_j \equiv \mathcal{Q}_j(z_j),
$$
where $z_j \in \Psi(\Lambda,r_1)$. Set $f(z)=H(z,|Du(z)|)$ and fix
$$
\theta=\frac{1+\theta_*}{2},\qquad
\theta_*=\max\left\{\frac n{n+2},\frac{q-1}{p}\right\}<1.
$$
This exponent is admissible in both phases. Applying Lemmas \ref{lem : reverse Holder inequality for p phase} and \ref{lem : reverse Holder inequality for p,q phase}, with the dilation factor absorbed into $c=c(\operatorname{data}_\gamma)>1$, gives
$$
\iints{\kappa\mathcal Q_j}f\,dz
\leq c\Lambda^{1-\theta}\iints{\mathcal Q_j\cap\{f>\Lambda/c\}}f^\theta\,dz.
$$
These are ordinary integrals, so that they can be summed over the pairwise disjoint cylinders $\mathcal Q_j$. The dilated cylinders cover $\Psi(\Lambda,r_1)$ up to a null set and are contained in $Q_{r_2}(z_0)$. Therefore,
\begin{equation}\label{eq : superlevel integral estimate}
\iints{Q_{r_1}(z_0)\cap\{f>\Lambda\}}f\,dz
\leq c\Lambda^{1-\theta}\iints{Q_{r_2}(z_0)\cap\{f>\Lambda/c\}}f^\theta\,dz
\end{equation}
whenever
$$
\Lambda>L_{12}:=\left(\frac{4\kappa r}{r_2-r_1}\right)^\beta\Lambda_0,
\qquad \beta=\frac{q(n+2)}2.
$$

To avoid assuming higher integrability in advance, let $N>0$, $f_N=\min\{f,N\}$, and
$$
F_N(s)=\iints{Q_s(z_0)}f f_N^\varepsilon\,dz,
\qquad 0<\varepsilon\leq1.
$$
This is finite because $f\in L^1$. Multiplying \eqref{eq : superlevel integral estimate} by $\varepsilon\Lambda^{\varepsilon-1}$, integrating from $L_{12}$ to $N$ when $N>L_{12}$, and using Tonelli's theorem yields
$$
F_N(r_1)\leq L_{12}^\varepsilon\iints{Q_{2r}(z_0)}f\,dz
+\frac{C\varepsilon}{1-\theta}F_N(r_2),
$$
with $C=C(\operatorname{data}_\gamma)$ independent of $N$ and $\varepsilon\in(0,1]$. Indeed, the inner integral on the right is bounded by
$$
\frac{\min\{N,cf\}^{1+\varepsilon-\theta}}{1+\varepsilon-\theta},
\qquad
f^\theta\min\{N,cf\}^{1+\varepsilon-\theta}
\leq c^{1+\varepsilon-\theta}f f_N^\varepsilon.
$$
If $N\leq L_{12}$, the same estimate follows directly from $f_N\leq L_{12}$. Choose $\varepsilon_0\leq1$ so that $C\varepsilon_0/(1-\theta)\leq1/2$. Since
$L_{12}^\varepsilon\leq(4\kappa r/(r_2-r_1))^\beta\Lambda_0^\varepsilon$, Lemma \ref{lem : a standard iteration lemma}, applied to the bounded function $F_N$ on $[r,2r]$, gives
$$
F_N(r)\leq c\Lambda_0^\varepsilon\iints{Q_{2r}(z_0)}f\,dz.
$$
Letting $N\to\infty$ by monotone convergence and dividing by $|Q_r(z_0)|$, we obtain
\begin{equation}\label{eq : H estimate in Q_r to Q_2r}
    \miint{Q_r(z_0)}H(z,|Du|)^{1+\varepsilon}\,dz
    \leq c\Lambda_0^\varepsilon\miint{Q_{2r}(z_0)}H(z,|Du|)\,dz
\end{equation}
for every $\varepsilon\in(0,\varepsilon_0)$, where both $c$ and $\varepsilon_0$ depend only on $\operatorname{data}_\gamma$.
Since $\lambda_0\geq 1$ and $2\leq p<q$, $\Lambda_0^\varepsilon\leq c \lambda_0^{\varepsilon q}$ for some $c=c(\operatorname{data}_\gamma,\|a\|_{L^\infty(\Omega_T)})$. Then the definition of $\lambda_0$ implies that
\begin{equation}\label{eq : Lambda_0 estimate}
    \Lambda_0^\varepsilon\miint{Q_{2r}(z_0)} H(z,|Du|)\, dz\leq c\left(\miint{Q_{2r}(z_0)} H(z,|Du|)\, dz\right)^{1+\frac{\varepsilon q}{2}}+c.
\end{equation}
Combining \eqref{eq : H estimate in Q_r to Q_2r} and \eqref{eq : Lambda_0 estimate} proves Theorem \ref{thm : main theorem with gamma}. \qquad \qquad \qquad \qquad \qquad \qquad $\Box$

\vspace{0.5cm}
\noindent\textbf{Acknowledgments.} This work is supported by the National Research Foundation of Korea (NRF) grant funded by the Korea government [Grant Nos. RS-2025-00555316 and RS-2025-25415411].


\begin{thebibliography}{10}

\bibitem{Acerbi2004}
E.~Acerbi, G.~Mingione, and G.~A. Seregin.
\newblock Regularity results for parabolic systems related to a class of
  non-{N}ewtonian fluids.
\newblock {\em Ann. Inst. H. Poincar\'{e} C Anal. Non Lin\'{e}aire},
  21(1):25--60, 2004.

\bibitem{Arora2023}
R.~Arora and S.~Shmarev.
\newblock Double-phase parabolic equations with variable growth and nonlinear
  sources.
\newblock {\em Adv. Nonlinear Anal.}, 12(1):304--335, 2023.

\bibitem{Baasandorj2020}
S.~Baasandorj, S.-S. Byun, and J.~Oh.
\newblock Calder\'{o}n-{Z}ygmund estimates for generalized double phase
  problems.
\newblock {\em J. Funct. Anal.}, 279(7):108670, 57, 2020.

\bibitem{Bahrouni2019}
A.~Bahrouni, V.~D. R{\u{a}}dulescu, and D.~D. Repov{\v{s}}.
\newblock Double phase transonic flow problems with variable growth: nonlinear
  patterns and stationary waves.
\newblock {\em Nonlinearity}, 32(7):2481--2495, 2019.

\bibitem{Baroni2015}
P.~Baroni, M.~Colombo, and G.~Mingione.
\newblock Harnack inequalities for double phase functionals.
\newblock {\em Nonlinear Anal.}, 121:206--222, 2015.

\bibitem{Baroni2018}
P.~Baroni, M.~Colombo, and G.~Mingione.
\newblock Regularity for general functionals with double phase.
\newblock {\em Calc. Var. Partial Differential Equations}, 57(2):Paper No. 62,
  48, 2018.

\bibitem{Benci2000}
V.~Benci, P.~D'Avenia, D.~Fortunato, and L.~Pisani.
\newblock Solitons in several space dimensions: {D}errick's problem and
  infinitely many solutions.
\newblock {\em Arch. Ration. Mech. Anal.}, 154(4):297--324, 2000.

\bibitem{Buryachenko2022}
K.~O. Buryachenko and I.~I. Skrypnik.
\newblock Local continuity and {H}arnack's inequality for double-phase
  parabolic equations.
\newblock {\em Potential Anal.}, 56(1):137--164, 2022.

\bibitem{Byun2021a}
S.-S. Byun and H.-S. Lee.
\newblock Calder\'{o}n-{Z}ygmund estimates for elliptic double phase problems
  with variable exponents.
\newblock {\em J. Math. Anal. Appl.}, 501(1):Paper No. 124015, 31, 2021.

\bibitem{Byun2021}
S.-S. Byun and H.-S. Lee.
\newblock Gradient estimates of {$\omega$}-minimizers to double phase
  variational problems with variable exponents.
\newblock {\em Q. J. Math.}, 72(4):1191--1221, 2021.

\bibitem{Byun2017}
S.-S. Byun and J.~Oh.
\newblock Global gradient estimates for non-uniformly elliptic equations.
\newblock {\em Calc. Var. Partial Differential Equations}, 56(2):Paper No. 46,
  36, 2017.

\bibitem{Byun2020}
S.-S. Byun and J.~Oh.
\newblock Regularity results for generalized double phase functionals.
\newblock {\em Anal. PDE}, 13(5):1269--1300, 2020.

\bibitem{Charkaoui2024}
A.~Charkaoui, A.~Ben-Loghfyry, and S.~Zeng.
\newblock Nonlinear parabolic double phase variable exponent systems with
  applications in image noise removal.
\newblock {\em Appl. Math. Model.}, 132:495--530, 2024.

\bibitem{Chen2006}
Y.~Chen, S.~Levine, and M.~Rao.
\newblock Variable exponent, linear growth functionals in image restoration.
\newblock {\em SIAM J. Appl. Math.}, 66(4):1383--1406, 2006.

\bibitem{cherfils2005stationary}
L.~Cherfils and Y.~Il'yasov.
\newblock On the stationary solutions of generalized reaction diffusion
  equations with $p\,$\&$\,q$-{L}aplacian.
\newblock {\em Commun. Pure Appl. Anal.}, 4(1):9--22, 2005.

\bibitem{Chlebicka2025}
I.~Chlebicka, P.~Garain, and W.~Kim.
\newblock Gradient higher integrability of bounded solutions to parabolic
  double-phase systems.
\newblock arXiv preprint arXiv:2512.11294, 2025.

\bibitem{Chlebicks2019}
I.~Chlebicka, P.~Gwiazda, and A.~Zatorska-Goldstein.
\newblock Parabolic equation in time and space dependent anisotropic
  {M}usielak-{O}rlicz spaces in absence of {L}avrentiev's phenomenon.
\newblock {\em Ann. Inst. H. Poincar\'{e} C Anal. Non Lin\'{e}aire},
  36(5):1431--1465, 2019.

\bibitem{Colombo2015a}
M.~Colombo and G.~Mingione.
\newblock Bounded minimisers of double phase variational integrals.
\newblock {\em Arch. Ration. Mech. Anal.}, 218(1):219--273, 2015.

\bibitem{Colombo2015}
M.~Colombo and G.~Mingione.
\newblock Regularity for double phase variational problems.
\newblock {\em Arch. Ration. Mech. Anal.}, 215(2):443--496, 2015.

\bibitem{Colombo2016}
M.~Colombo and G.~Mingione.
\newblock Calder\'{o}n-{Z}ygmund estimates and non-uniformly elliptic
  operators.
\newblock {\em J. Funct. Anal.}, 270(4):1416--1478, 2016.

\bibitem{DeFilippis2019}
C.~De~Filippis and G.~Mingione.
\newblock A borderline case of {C}alder\'{o}n-{Z}ygmund estimates for
  nonuniformly elliptic problems.
\newblock {\em St. Petersburg Math. J.}, 31(3):455--477, 2020.

\bibitem{1993_Degenerate_parabolic_equations_DiBenedetto}
E.~DiBenedetto.
\newblock {\em Degenerate parabolic equations}.
\newblock Universitext. Springer-Verlag, New York, 1993.

\bibitem{Esposito2004}
L.~Esposito, F.~Leonetti, and G.~Mingione.
\newblock Sharp regularity for functionals with {$(p,q)$} growth.
\newblock {\em J. Differential Equations}, 204(1):5--55, 2004.

\bibitem{2003_Giusti_Direct_methods_in_the_calculus_of_variations}
E.~Giusti.
\newblock {\em Direct methods in the calculus of variations}.
\newblock World Scientific Publishing Co., Inc., River Edge, NJ, 2003.

\bibitem{Harjulehto2021}
P.~Harjulehto and P.~H\"ast\"o.
\newblock Double phase image restoration.
\newblock {\em J. Math. Anal. Appl.}, 501(1):Paper No. 123832, 12, 2021.

\bibitem{Harjulehto2013}
P.~Harjulehto, P.~H\"ast\"o, V.~Latvala, and O.~Toivanen.
\newblock Critical variable exponent functionals in image restoration.
\newblock {\em Appl. Math. Lett.}, 26(1):56--60, 2013.

\bibitem{Hasto_2021}
P.~H\"{a}st\"{o} and J.~Ok.
\newblock Higher integrability for parabolic systems with {O}rlicz growth.
\newblock {\em J. Differential Equations}, 300:925--948, 2021.

\bibitem{Haestoe2022a}
P.~H\"{a}st\"{o} and J.~Ok.
\newblock Maximal regularity for local minimizers of non-autonomous
  functionals.
\newblock {\em J. Eur. Math. Soc. (JEMS)}, 24(4):1285--1334, 2022.

\bibitem{Haestoe2022}
P.~H\"{a}st\"{o} and J.~Ok.
\newblock Regularity theory for non-autonomous partial differential equations
  without {U}hlenbeck structure.
\newblock {\em Arch. Ration. Mech. Anal.}, 245(3):1401--1436, 2022.

\bibitem{Kbiri2014}
M.~Kbiri~Alaoui, T.~Nabil, and M.~Altanji.
\newblock On some new non-linear diffusion models for the image filtering.
\newblock {\em Appl. Anal.}, 93(2):269--280, 2014.

\bibitem{Kim2024}
B.~Kim and J.~Oh.
\newblock Higher integrability for weak solutions to parabolic multi-phase
  equations.
\newblock {\em J. Differential Equations}, 409:223--298, 2024.

\bibitem{KimOh2026}
B.~Kim and J.~Oh.
\newblock Bounded solutions and interpolative gap bounds for degenerate
  parabolic double phase problems.
\newblock arXiv preprint arXiv:2511.13454, 2025.

\bibitem{Kim2025}
B.~Kim, J.~Oh, and A.~Sen.
\newblock Parabolic {L}ipschitz truncation for multi-phase problems: the
  degenerate case.
\newblock {\em Adv. Calc. Var.}, 18(3):979--1010, 2025.

\bibitem{Wontae2023b}
W.~Kim.
\newblock Calder\'on-{Z}ygmund type estimate for the singular parabolic
  double-phase system.
\newblock {\em J. Math. Anal. Appl.}, 551(1):Paper No. 129593, 33, 2025.

\bibitem{2023_Gradient_Higher_Integrability_for_Degenerate_Parabolic_Double-Phase_Systems}
W.~Kim, J.~Kinnunen, and K.~Moring.
\newblock Gradient higher integrability for degenerate parabolic double-phase
  systems.
\newblock {\em Arch. Ration. Mech. Anal.}, 247(5):Paper No. 79, 46, 2023.

\bibitem{Wontae2023a}
W.~Kim, J.~Kinnunen, and L.~S\"arki\"o.
\newblock Lipschitz truncation method for parabolic double-phase systems and
  applications.
\newblock {\em J. Funct. Anal.}, 288(3):Paper No. 110738, 60, 2025.

\bibitem{Wontae2025}
W.~Kim, K.~Moring, and L.~S\"arki\"o.
\newblock H\"older regularity for degenerate parabolic double-phase equations.
\newblock {\em J. Differential Equations}, 434:Paper No. 113231, 34, 2025.

\bibitem{Wontae2024}
W.~Kim and L.~S\"arki\"o.
\newblock Gradient higher integrability for singular parabolic double-phase
  systems.
\newblock {\em NoDEA Nonlinear Differential Equations Appl.}, 31(3):Paper No.
  40, 38, 2024.

\bibitem{Kinnunen2000}
J.~Kinnunen and J.~L. Lewis.
\newblock Higher integrability for parabolic systems of {$p$}-{L}aplacian type.
\newblock {\em Duke Math. J.}, 102(2):253--271, 2000.

\bibitem{Fang2010}
F.~Li, Z.~Li, and L.~Pi.
\newblock Variable exponent functionals in image restoration.
\newblock {\em Appl. Math. Comput.}, 216(3):870--882, 2010.

\bibitem{Marcellini1989}
P.~Marcellini.
\newblock Regularity of minimizers of integrals of the calculus of variations
  with nonstandard growth conditions.
\newblock {\em Arch. Ration. Mech. Anal.}, 105(3):267--284, 1989.

\bibitem{Marcellini1991}
P.~Marcellini.
\newblock Regularity and existence of solutions of elliptic equations with
  {$p,q$}-growth conditions.
\newblock {\em J. Differential Equations}, 90(1):1--30, 1991.

\bibitem{Mingione2021}
G.~Mingione and V.~R\v{a}dulescu.
\newblock Recent developments in problems with nonstandard growth and
  nonuniform ellipticity.
\newblock {\em J. Math. Anal. Appl.}, 501(1):Paper No. 125197, 41, 2021.

\bibitem{Ok2017}
J.~Ok.
\newblock Regularity of {$\omega$}-minimizers for a class of functionals with
  non-standard growth.
\newblock {\em Calc. Var. Partial Differential Equations}, 56(2):Paper No. 48,
  31, 2017.

\bibitem{Ok2020}
J.~Ok.
\newblock Regularity for double phase problems under additional integrability
  assumptions.
\newblock {\em Nonlinear Anal.}, 194:111408, 13, 2020.

\bibitem{Sen2025}
A.~Sen.
\newblock Gradient higher integrability for degenerate/singular parabolic
  multi-phase problems.
\newblock {\em J. Geom. Anal.}, 35(6):Paper No. 170, 95, 2025.

\bibitem{Singer2016}
T.~Singer.
\newblock Existence of weak solutions of parabolic systems with {$p,
  q$}-growth.
\newblock {\em Manuscripta Math.}, 151(1-2):87--112, 2016.

\bibitem{Zhikov1986}
V.~V. Zhikov.
\newblock Averaging of functionals of the calculus of variations and elasticity
  theory.
\newblock {\em Izv. Akad. Nauk SSSR Ser. Mat.}, 50(4):675--710, 877, 1986.

\bibitem{Zhikov1993}
V.~V. Zhikov.
\newblock Lavrentiev phenomenon and homogenization for some variational
  problems.
\newblock {\em C. R. Acad. Sci. Paris S\'{e}r. I Math.}, 316(5):435--439, 1993.

\bibitem{Zhikov1995}
V.~V. Zhikov.
\newblock On {L}avrentiev's phenomenon.
\newblock {\em Russian J. Math. Phys.}, 3(2):249--269, 1995.

\bibitem{Zhikov1997}
V.~V. Zhikov.
\newblock On some variational problems.
\newblock {\em Russian J. Math. Phys.}, 5(1):105--116 (1998), 1997.

\end{thebibliography}
\end{document}